\documentclass[10pt]{article}

\usepackage{amssymb,latexsym,amsmath,amsthm,verbatim}
\usepackage{graphicx,epsfig,epstopdf,amssymb,color}

\newcommand{\eps}{\epsilon}
\newcommand{\del}{\partial}
\newcommand{\be}{\begin{equation}}
\newcommand{\ee}{\end{equation}}
\newcommand{\bea}{\begin{eqnarray}}
\newcommand{\eea}{\end{eqnarray}}

\newcommand{\ba}{\begin{array}}
\newcommand{\ea}{\end{array}}
\newcommand{\R}{\mathbb{R}}
\newcommand{\Z}{\mathbb{Z}}
\newcommand{\N}{\mathbb{N}}

\renewcommand{\:}{\thinspace :}

\newtheorem{theorem}{Theorem}[section]
\newtheorem{lemma}[theorem]{Lemma}
\newtheorem{proposition}[theorem]{Proposition}

\theoremstyle{definition}

\newtheorem{remark}[theorem]{Remark}

\newtheorem{hyp}[theorem]{Hypothesis}

\begin{document}
\baselineskip=14pt
\title{Existence and uniqueness of traveling waves in a class of unidirectional lattice differential equations }

\author{Aaron Hoffman and  Benjamin Kennedy}

\date{}

\maketitle

\renewcommand{\thefootnote}{ }

\begin{abstract}
We prove the existence and uniqueness, for wave speeds sufficiently large, of monotone traveling wave solutions connecting stable to unstable spatial equilibria for a class of $N$-dimensional lattice differential equations with unidirectional coupling.  This class of lattice equations includes some cellular neural networks, monotone systems, and semi-discretizations for hyperbolic conservation laws with a source term.  We obtain a variational characterization of the critical wave speed above which monotone traveling wave solutions are guaranteed to exist.  We also discuss non-monotone waves, and the coexistence of monotone and non-monotone waves. 
\end{abstract}

\section{Introduction} \label{INT}

Let $\Z^D$ be the $D$-dimensional integer lattice, and let $p := (p_1,\ldots,p_D)$ be a point of the lattice.  We choose and fix a vector $\xi = (\xi_1,\ldots,\xi_D) \in \R^D$ such that $\|\xi\| = 1$.  We define a map 
\[
\chi : \Z^D \to (\Z^D)^{N+1}, \ \ \chi(p) = (p,\chi_1(p),\ldots,\chi_N(p))
\]
satisfying the following: first, the translation-invariance condition
\[
\chi(p + q) = \chi(p) + q
\]
holds for all $p,q \in \Z^D$; second, all points $\chi_i(p) \in  \Z^D$ lie on one side of the hyperplane through $p$ determined by $\xi$:
\begin{equation}\label{eq:unidir}
\xi \cdot (p - \chi_i(p)) \geq 0, \ \ i \in \{1,\ldots,N\}.  
\end{equation}
We shall refer to this latter condition as \emph{unidirectionality} (with respect to $\xi$).  If we take $D = 2$, $N=3$, and $\xi  = (1,1)$, for example, the map 
\[
\chi((k,j)) = ((k,j),(k-1,j),(k,j-1),(k-1,j-1))
\]
satisfies both the translation-invariance and unidirectionality conditions.  

In this paper we study differential equations on $\Z^D$ with first-order unidirectional coupling: to each $p \in \Z^D$ we associate a function $u_p(\cdot) : \R \to \R$, and consider equations of the form
\begin{equation}\label{MODEQ}
u_p'(t) = g(u_{p}(t),u_{\chi_1(p)}(t),\ldots,u_{\chi_N(p)}(t)). 
\end{equation}\noindent

Equation $\eqref{MODEQ}$ is an example of a lattice differential equation (LDE).  LDEs arise in diverse applications and have been extensively studied (see, for example, \cite{chua:1988a,chua:1988b,cahn:1999,keener:1987,weinberger:1982,mallet-paret:2003,bates:2003} and the references therein).  The unidirectional coupling in \eqref{MODEQ} --- that is, the requirement that all of the lattice points determining $u_p'(t)$ lie on one side of a hyperplane through $p$ --- admittedly makes \eqref{MODEQ} into a specialized class of problems.   Such equations do, however, arise: in addition to models where spatial interactions are genuinely discrete and unidirectional, \eqref{MODEQ} includes discretizations of certain first-order PDEs and models in which the lattice sites represent agents or categories (rather than spatial locations).  We discuss such examples at greater length below in Section \ref{subsec:ex}.

A \emph{traveling wave solution of speed $c$ in direction $\xi$} of \eqref{MODEQ} is a solution that satisfies, for all lattice points $p$ and times $t$,
\begin{equation}
u_p(t) = \phi(\xi \cdot p  - ct), \label{eq:ansatzN}
\end{equation}
where $c > 0$, $\xi$ is as in \eqref{eq:unidir}, and the \emph{wave profile} $\phi$ is a function from $\R$ to $\R$.  We shall adhere to this notation throughout, and shall refer to a traveling wave solution of speed $c$ in direction $\xi$ simply as a \emph{traveling wave solution}.  

All instances of \eqref{MODEQ} we consider here will have the equilibrium solutions
\[
u_p(t) \equiv 0 \ \forall \ t,p \ \ \mbox{and} \ \ u_p(t) \equiv 1 \ \forall \ t,p.
\]
If the wave profile $\phi$ for a traveling wave solution satisfies the boundary conditions
\[
\phi(x) \to 1 \ \mbox{as} \ x \to -\infty \ \mbox{and} \ \phi(x) \to 0 \ \mbox{as} \ x \to \infty
\]
we shall say that the solution \emph{connects $1$ to $0$} (though observe, in this case, that $u_p(-\infty) = 0$ and $u_p(\infty) = 1$ for all $p \in \Z^D$).   If $\phi(x)$ is strictly decreasing for all $x \in \R$, we say that the solution is \emph{decreasing}.  If $\phi(x)$ is strictly decreasing for all $x$ less than some $x_0 \in \R$, we say that $\phi$ is \emph{initially decreasing}.  For brevity, we shall often refer to traveling wave solutions connecting $1$ to $0$ as \emph{fronts}.  Abusing terminology, we shall refer to the corresponding wave profile $\phi$ as a front as well.  It is the fronts --- the traveling wave solutions connecting $1$ to $0$ --- that occupy most of our attention in this paper.  

If $\phi : \R \to \R$ is the wave profile for a traveling wave solution of \eqref{MODEQ}, then $\phi$ solves the \emph{wave profile equation} 
\begin{eqnarray} 
c\phi'(x) = - g(\phi(x),\phi(x - \kappa_1),\ldots,\phi(x-\kappa_N)), \label{eq:wpe} 
\end{eqnarray}
where $\kappa_i = \xi \cdot (p - \chi_i(p))$ for $i = 1,\ldots, N$.  Conversely, if $\phi$ satisfies $\eqref{eq:wpe}$, then $u_p(t) = \phi(\xi \cdot p - ct)$ satisfies $\eqref{MODEQ}$.  The unidirectionality condition \eqref{eq:unidir} implies that each $\kappa_i$ is nonnegative, and so \eqref{eq:wpe} is a delay equation (rather than an equation of mixed type).  The relative tractability of \eqref{eq:wpe} in this case is the primary reason that we consider unidirectional models.  By modifying the function $g$ and the constant $N$ we can, without loss of generality, assume that $\kappa_i > 0$ for $i \ge 1$, and we do so henceforth.   

If a particular function $\phi : \R \to \R$ is understood, when convenient we shall write
\[
\Phi(x) = (\phi(x),\phi(x - \kappa_1),\ldots,\phi(x - \kappa_N)) \in \R^{N+1},  
\]
and shall write $\eqref{eq:wpe}$ in the more compact form 
\begin{equation*}
c\phi'(x) = - g(\Phi(x)).
\end{equation*}

Since we are assuming that the spatial equilibria $0$ and $1$ solve \eqref{MODEQ}, the constant functions $0$ and $1$ are equilibrium solutions of \eqref{eq:wpe} as well.  The traveling wave solution connects $1$ to $0$ if and only if $\phi$ satisfies the boundary conditions
\begin{eqnarray}
\phi(-\infty) = 1; \label{eq:negbc} \\
\phi(\infty) = 0. \label{eq:posbc}
\end{eqnarray}

In this paper we study the existence, uniqueness, and monotonicity of traveling wave solutions of \eqref{MODEQ} connecting $1$ to $0$.  This amounts to considering the existence, uniqueness, and monotonicity of solutions of the boundary value delay equation $(\ref{eq:wpe},\ref{eq:negbc},\ref{eq:posbc})$.  (Since any time-translation of a traveling wave solution is still a traveling wave solution, by uniqueness we shall always mean uniqueness up to time translation --- equivalently, uniqueness of $\phi$ up to translation in $x$.)  Our basic approach for existence and uniqueness is to exploit the one-dimensionality of the unstable manifold at $1$ of \eqref{eq:wpe}, denoted $W^u(1)$.  We show that, for $c$ sufficiently large, the initially decreasing branch of $W^u(1)$ remains monotone forever, and is in fact a front.  Using an Implicit Function Theorem argument, we show that the set $\mathcal{U}$ of wave speeds $c$ for which fronts exist is open and, moreover, that some of the corresponding fronts are non-monotone.  We also explore the behavior of waves as $c$ approaches and reaches the boundary of $\mathcal{U}$.   

To make use of results on invariant manifolds for \eqref{eq:wpe}, we shall need to work in the standard phase space for delay equations: with $r = \max_j(\kappa_j)$, we write $C = C([-r,0],\R)$ for the Banach space of continuous real-valued functions from $[-r,0]$ to $\R$, equipped with the sup norm.  If $y$ is any continuous function defined on $[\tau - r,\tau]$, in the usual way we write $y_\tau$ for the member of $C$ given by
\[
y_\tau(s) = y(\tau + s), \ s \in [-r,0].
\] 
Given $y_0 \in C$, our hypotheses on $g$ will guarantee that $y_0$ extends uniquely to a function $y: [-r,T) \to \R$ such that $y$ solves \eqref{eq:wpe} for all $t > 0$ and $T \leq \infty$ is maximal (see, for example, Chapter 2 of \cite{hale:1993}).    
 
We now describe the conditions that we impose on the feedback function $g$ in \eqref{MODEQ}.   We write the general point of $\R^{N+1}$ as $s = (s_0,s_1,\ldots,s_N)$, and for $a \in \R$ we write ${\bf a} = (a,a,\ldots,a) \in \R^{N+1}$.  The reader may find it helpful to keep the following example in mind, which furnishes a primary motivation for this paper: $D = 1$, $N = 1$, $\xi = 1$, and, for all $p \in \Z$,  
\begin{equation}\label{eq:vzexcov1}
u'_p = -u_p + 2u_{p-1} - u_{p-1}^2. 
\end{equation}
\eqref{eq:vzexcov1} arises in a clock model for a dilute gas described in \cite{vanzon:1998}.  We shall discuss this equation further in section \ref{subsec:ex}.  In the case of \eqref{eq:vzexcov1}, $g$ becomes 
\[
g: \R^2 \to \R, \ \ g(s_0,s_1) = -s_0 + 2s_1 - s_1^2.
\]
The following hypotheses (G1) and (G2) serve to generalize certain features of this motivating function $g$.   Hypothesis (G3) will apply to certain examples that we consider in Section \ref{EXAMPLES}.

\begin{hyp}\label{hyp:hyp}
$g: \R^{N+1} \to \R$ is a locally Lipschitz function, $C^1$ on neighborhoods of ${\bf 0}$ and ${\bf 1}$.  $g$ satisfies the following:
\begin{itemize}
\item ({\bf G1})
\begin{itemize}
\item ({\bf G1.1}) $g({\bf 0}) = g({\bf 1}) = 0$; furthermore $g({\bf a}) \ne 0$ for any $a \in (0,1)$.
\item ({\bf G1.2}) $\frac{\partial g}{\partial s_0}({\bf a}) < 0$ for ${\bf a} = {\bf 0},{\bf 1}$; $\frac{\partial g}{\partial s_i}({\bf a}) \geq 0$ for $i = 1,\ldots, N$ and for ${\bf a} = {\bf 0},{\bf 1}$; and 
\[
\sum_{i=0}^N \frac{\partial g}{\partial s_i}({\bf 1})  < 0 < \sum_{i=0}^N \frac{\partial g}{\partial s_i}({\bf 0}).
\] 
\item ({\bf G1.3}) There is some  $\beta \in \R^{N+1}$ with $\sum_{i=0}^N \beta_i > 0$, $\beta_i > 0$ for $i = 1,\ldots, N$, and $\beta_0 < 0$, such that $0 < g(s) < \sum_{i=0}^N \beta_i s_i$ whenever 
\[
0 < s_0 < s_i < 1 \ \mbox{for all} \  i = 1, \ldots, N.
\]
\end{itemize}
\item ({\bf G2}) If $1 < \min_i s_i$ and $s_0 = \max_i s_i$, then $g(s) < 0$.
\item ({\bf G3}) There are constants $M_0 > 0$ and $\eta \in (0,1)$ such that, whenever $|s_0| = \max_i|s_i|$ and $|s_0| \geq M_0$, then  $\frac{|g(s) + s_0|}{|s_0|} \leq \eta$.
\end{itemize}
\end{hyp}

Conditions (G1) restrict the feedback function $g$ in the cube $[0,1]^{N+1}$.  (G1.1) ensures that $0$ and $1$ are equilibria of \eqref{MODEQ} and \eqref{eq:wpe}.  It further ensures that there are no other equilibria between $0$ and $1$ at which a wave emanating from $1$ might get stuck. We will show below that (G1.2) guarantees that $1$ is a hyperbolic equilibrium with a one-dimensional unstable manifold and that, for $c$ large enough, the dominant eigenvalue for the linearization of \eqref{eq:wpe} about $0$ is real and negative.  (G1.3) guarantees that any initially decreasing solution of $(\ref{eq:wpe},\ref{eq:negbc})$ remains strictly decreasing at least until it crosses zero (see Lemma $\ref{lem:decrease}$); moreover, for $c$ large enough relative to $\beta$, such a solution cannot overshoot $0$ (see Proposition $\ref{prop:exist}$).  All the conditions (G1) together will thus guarantee that, for $c$ large enough, the initially decreasing branch of the unstable manifold at $1$ is in fact a strictly decreasing front.  The second inequality in condition (G1.3) can always be satisfied with some $\beta$; the idea is, given $g$, to choose $\beta$ to yield maximal information about the range of $c$ for which decreasing fronts exist.  

Conditions (G2) and (G3) restrict the coupling function $g$ away from the cube $[0,1]^N$.  We have given conditions on $g$ that are easy to verify, but our application of these conditions is embodied in the following two lemmas; in our main theorems, (G2) and (G3) can replaced by any conditions that make the conclusions of Lemmas \ref{lem:G21} and \ref{lem:G3}, respectively, hold.  

\begin{lemma}\label{lem:G21}
Suppose that (G2) holds.  Then solutions of $\eqref{eq:wpe}$ with initial condition $\phi_0 \in C[-r,0]$ satisfying 
\[
 1 < \phi_0(s) \leq \phi_0(0) \ \mbox{for all} \ s \in [-r,0]
\]
are strictly increasing for $t > 0$.
\end{lemma}

\begin{proof}
Since 
\[
c\phi'(0) = -g(\phi(0),\phi(-\kappa_1),\ldots,\phi(-\kappa_N))
\]
and 
\[
1 < \min_i \phi(-\kappa_i) \leq \max_i \phi(-\kappa_i) \leq \phi(0)
\]
by hypothesis, (G2) implies that $\phi'(0) > 0$.  Imagine that there is some first positive time $\tau > 0$ such that $\phi'(\tau) = 0$.  Then, since $\phi$ is strictly increasing on $[0,\tau]$, we have 
\[
1 < \min_i \phi(\tau - \kappa_i) \leq \max_i \phi(\tau-\kappa_i) \leq \phi(\tau),
\]
whence $\phi'(\tau) > 0$ --- a contradiction.  
\end{proof}

As we shall see, assuming that (G1) is satisfied, (G2) is a sufficient condition for uniqueness of fronts because it prevents the initially increasing branch of the unstable manifold at $1$ from approaching zero.  

(G3) ensures that any solution of \eqref{eq:wpe} that satisfies the boundary condition $\eqref{eq:negbc}$ and does not approach $\pm \infty$ remains bounded by some uniform constant $M$, as we now explain.  
\begin{lemma} \label{lem:G3}
Suppose that (G3) holds.  Then there is a positive constant $M$ such that any solution of \eqref{eq:wpe} that satisfies $\eqref{eq:negbc}$ and does not approach $\pm \infty$ as $x \to \infty$ remains bounded by $M$.  $M$ does not depend on $c > 0$.  
\end{lemma}
\begin{proof} 
Choose $M_0 > 1$ such that $\max_i |s_i| = |s_0| \geq M_0$ implies that $|g(s) + s_0| \leq \eta |s_0|$, for some $\eta \in (0,1)$.  Now take $M \geq M_0$.

Suppose that $\sup_{x \in \R} |\phi(x)| > M$, and set 
\[
x_0 = \inf \{ \ x \in \R \ : \ |\phi(x)| \geq M \ \}
\]
(and so $|\phi(x_0)| = M$).   Then $|\phi(x_0 - \kappa_i)| \leq M = |\phi(x_0)|$ for all $i$ and so $|g(\Phi(x_0)) + \phi(x_0)| \le \eta |\phi(x_0)|  < |\phi(x_0)|$.  Thus $\phi(x_0)$ and $g(\Phi(x_0))$ necessarily have different signs with $|g(\Phi(x_0))| \geq (1-\eta)M \geq (1-\eta)M_0$.  

Thus $|\phi|$ is strictly increasing at $x = x_0$, with $|\phi(x_0)|' \geq \frac{1}{c}(1-\eta)M_0$.  A similar argument shows that $|\phi(x_0)|$ is strictly increasing, with derivative bounded away from zero, whenever $x_0$ is the first time that $|\phi(x_0)| = \bar{M}$, for any $\bar{M} \geq M$.  We conclude that, if $|\phi(x)|$ is not bounded by $M$, then $\phi(\infty) = \pm \infty$.  
\end{proof}

To complete the preparation necessary to state our main theorems, we define some quantities related to the linearizations of \eqref{eq:wpe} about its equilibria.  Let $c > 0$ and $\alpha \in \R^{N+1}$ be given, with $\alpha_i \geq 0$ for $i = 1,\ldots, N$.  Recall that $\kappa = (\kappa_1,\ldots,\kappa_N)$ with $\kappa_i > 0$ for $i = 1, \ldots, N$ is the vector of delays in equation $\eqref{eq:wpe}$.  We define the following function of the complex number $\lambda$:
\begin{equation}
D(\lambda;c,\alpha) := c\lambda + \alpha_0 + \sum_{i=1}^N \alpha_i e^{-\kappa_i \lambda}.
\label{eq:Ddef}
\end{equation}
With ${\bf a} = (a,\ldots,a)$ equal to ${\bf 0}$ or ${\bf 1}$, $D(\lambda;c,\nabla g({\bf a}))=0$ is the characteristic equation for the linearization of $\eqref{eq:wpe}$ at the equilibrium $a$.  

With $\alpha$ as above, we define the number
\begin{equation}
c(\alpha) := \inf_{\lambda > 0} \frac{\alpha_0 + \sum_{i=1}^N \alpha_i e^{\kappa_i \lambda}}{\lambda}.
\label{eq:cdef}
\end{equation}
$c(\alpha)$ is finite if $\sum_{i=0}^N \alpha_i > 0$; in particular, (G1) implies that $c(\nabla g({\bf 0}))$ and $c(\beta)$ are finite.  As we shall see below, for $c > c(\nabla g({\bf 0}))$ the dominant root of the characteristic equation of $\eqref{eq:wpe}$ at $0$ is real and negative.  ($c(\alpha)$ is the so-called ``spreading speed" for the linear equation $u'_p = \alpha_0 u_p + \sum_{i=1}^N \alpha_i u_{\chi_i(p)}$ --- see \cite{weinberger:1982}.)  We will show that decreasing fronts necessarily exist for $c \geq c(\beta)$, where $\beta$ is as in (G1.3). 

We also define the number
\begin{equation}
b(\alpha) := \inf \{ \bar{c} >  0 \ : \ \mbox{all roots of $D(\lambda;c,\alpha)$ have negative real part if $c > \bar{c}$} \ \}.
\label{eq:cdef}
\end{equation}
As $c$ drops below $b(\nabla g({\bf 0}))$, the equilibrium solution $0$ of \eqref{eq:wpe} becomes unstable.  

We now state our main results.  We prove our results and present examples in subsequent sections.  We remind the reader by a ``front" we mean a traveling wave solution of \eqref{MODEQ} connecting $1$ to $0$ or a solution of (\ref{eq:wpe},\ref{eq:negbc},\ref{eq:posbc}).  

\begin{theorem}[Existence and uniqueness of initially decreasing fronts] \label{thm:main}
Let $\mathcal{U}$ be the subset of $c \in [b(\nabla g({\bf 0})),\infty)$ for which an initially decreasing front exists, and let $\mathcal{M} \subset \; \mathcal{U}$ be the subset of $c \in [b(\nabla g({\bf 0})),\infty)$ for which a strictly decreasing front exists.   

Assume that (G1) holds.  Then
\begin{enumerate}
\item $0 < b(\nabla g({\bf 0})) < c(\nabla g({\bf 0})) \leq c(\beta)$.
\item For every $c \in \mathcal{U}$ there is exactly one (up to translation) initially decreasing front.  If (G2) holds, this front is the unique front (not just the unique initially decreasing front). 
\item  $\mathcal{M}$ is closed and contains a maximal ray $[c_m,\infty)$, where $c_m \leq c(\beta)$.  Moreover, $\mathcal{M}$ does not intersect $[b(\nabla g({\bf 0})),c(\nabla(g({\bf 0})))$.
\item If $g$ is $C^1$ everywhere, $\mathcal{U}$ is open relative to $[b(\nabla g({\bf 0})),\infty)$, and contains a maximal ray with infimum $c_f$, where $c_f < c_m$.  
\end{enumerate}
\end{theorem}

\begin{remark}
In the case $\beta = \nabla g({\bf 0})$, part 3 of Theorem $\ref{thm:main}$ implies that $c_m = c(\beta)$.  Part 4 implies that, if $g$ is $C^1$, there are wave speeds for which non-monotone fronts exist.  We suspect that this smoothness assumption can be relaxed.   
\end{remark}

Theorems asserting the existence of a critical wave speed $c_*$ such that some monostable evolution equation has monotone traveling wave solutions of all speeds greater than $c_*$ have a history going back to the 1930s \cite{fisher:1937, KPP:1937}.  We do not attempt to survey this literature here.  We do mention, however, the distinction between the cases $c_m = c(\nabla g({\bf 0}))$ and $c_m > c(\nabla g({\bf 0}))$, which is discussed at length in  \cite{vansaarloos:2003}.  In the former case, the front $\phi_{c_m}$ travels at the spreading speed of $\eqref{MODEQ}$ linearized about zero.  For this reason it is called a ``pulled front'' because it appears to be pulled by its leading edge.  In the latter case, $\phi_{c_m}$ is called a ``pushed front'' because the speed is determined not by the leading edge of the front but by the behavior of feedback function $g$ in the interfacial region.  In our framework, we necessarily get pulled fronts when we can take $\beta = \nabla g({\bf 0})$, but when we cannot take $\beta = \nabla g({\bf 0})$, it might be that only pushed fronts are possible.  

Our second main theorem describes the behavior of fronts as $c$ approaches and reaches a point of $\del U$ ($c_f$ in particular), provided that this point is greater than $b(\nabla g({\bf 0}))$, that $g$ is $C^1$, and that (G3) holds.   

\begin{theorem}\label{thm:main2}
Suppose that $g$ is $C^1$ and that (G3) holds.  If $c \in \del \mathcal{U}$ and $c > b(\nabla g({\bf 0}))$, \eqref{eq:wpe} has two distinct bounded nonconstant solutions $\psi^0$ and $\psi^1$, neither of which is connects $1$ to $0$, such that $\psi^1(-\infty) = 1$ and $\psi^0(\infty) = 0$.  
\end{theorem}

Intuitively, the above theorem describes the splitting, as $c$ reaches $\del \mathcal{U}$ from the interior of $\mathcal{U}$, of a front --- a heteroclinic between equilibria --- into a pair of heteroclinics (the solutions $\psi^0$ and $\psi^1$) linking $0$ and $1$, respectively, to some other invariant set.  Under some additional hypotheses, we can make this idea more precise.   The following, for example, is an immediate consequence of Mallet-Paret and Sell's Poincar\'{e}-Bendixson Theorem for delay equations (see Theorem 2.1 in \cite{mallet-paret:1996}).

\begin{proposition}\label{prop:pb}
Suppose that $N = 1$, that $g: \R^2 \to \R$ is $C^1$ and everywhere strictly decreasing in its second coordinate, and that the equilibria of \eqref{eq:wpe} are all isolated.  Then given any bounded solution $\phi$ of $\eqref{eq:wpe}$ (in particular $\psi^1$) either the $\omega$-limit set of $\phi$ is a single periodic solution, or the $\omega$- and $\alpha$- limit sets of any solution in the $\omega$-limit set of $\phi$ consist of equilibria of $\eqref{eq:wpe}$. 
\end{proposition}

An open problem is what conditions on $g$ guarantee that $c_f = b(\nabla g({\bf 0}))$ --- that is, whether (roughly speaking) fronts persist as far as is suggested by the linearization of \eqref{eq:wpe} about $0$.  When $D = 1$, $N = 1$, and $g$ satisfies appropriate hypotheses (such as monotonicity, and conditions to rule out equilibria other than $1$ and $0$) Theorem $\ref{thm:main2}$ together with Proposition $\ref{prop:pb}$ reduce the proof that $c_f = b(\nabla g({\bf 0}))$ to showing that for $c > b(\nabla g({\bf 0}))$, $\eqref{eq:wpe}$ admits neither periodic solutions that oscillate about zero nor solutions that are homoclinic to $1$: for then the solution $\psi^1$ in Theorem $\ref{thm:main2}$ cannot exist for $c > b(\nabla g({\bf 0}))$.  Proving nonexistence of periodic solutions oscillating about $0$ for $\eqref{eq:wpe}$ when $c > b(\nabla g({\bf 0}))$ is obverse to the problem of proving the existence of periodic solutions oscillating about an unstable equilibrium.  This latter problem has been extensively studied for delay equations in the so-called negative feedback case \cite{nussbaum:1974,mallet-paret:1988,krisztin:2008,diekmann:1995}.

\subsection{Related work and examples} \label{subsec:ex}

In this section we describe some closely related previous work and present some applications of our results. The part of Theorem \ref{thm:main} pertaining to the existence of monotone fronts overlaps with some earlier results in the case that $g$ satisfies appropriate monotonicity conditions --- for example, that
\begin{equation}\label{eq:quasimon}
g(s_0,\ldots,s_N) \ \mbox{is increasing in $s_1,\ldots,s_N$ for} \ s \in [0,1]^{N+1}.
\end{equation}
The part of Theorem \ref{thm:main} pertaining to non-monotone waves echoes some earlier results for very special cases of (non-smooth) $g$.  The novelty of our work lies in the the relaxation of monotonicity conditions on $g$, and the extension of existence results for non-monotone waves to a (relatively) broad class of functions $g$.  

In \cite{weinberger:1982}, Weinberger considers difference equations of the form
\[
u(t+1,\cdot) = Q(u(t,\cdot)).
\]
Here the map $Q$ acts on functions $u$ which themselves map a spatial domain to the positive reals.  The spatial domain is $D$-dimensional and either continuous or discrete.  If condition \eqref{eq:quasimon} holds, time-$\tau$ maps for \eqref{MODEQ} fit into the framework of \cite{weinberger:1982}; in this case, the existence of decreasing fronts for \eqref{MODEQ} when $c \geq c(\beta)$ is a consequence of results in \cite{weinberger:1982}.   The main focus in \cite{weinberger:1982} is on the so-called asymptotic spreading speed of initial data supported on a compact set; the uniqueness of monotone fronts, and the existence and behavior of non-monotone fronts, are not directly addressed.  

The existence results in \cite{weinberger:1982} use a monotone iteration technique.  The same is true of \cite{hsu:2000} (described further below), where a monotonicity condition on $g$ is also imposed.  Recently such techniques have been extended, in the setting of lattice integro-difference equations, to the case where the nonlinearity is not necessarily monotone but satisfies conditions which are similar to our (G1).  In particular \cite{li:2009} and \cite{hsu:2008} both obtain the existence of (not necessarily monotone) traveling waves as well as a variational characterization of the minimum wave speed guaranteeing monotone fronts.  Although our setting and techniques differ, our results here can be regarded as complementing these latter works.

We now turn to some more specific applications of \eqref{MODEQ}.

{\flushleft {\bf Cellular Neural Networks}}

In \cite{hsu:2000} Hsu and Lin consider lattice equations of the form 
\begin{equation}\label{eq:cnn} 
u_n' = -u_n + z+ \sum_{\ell = 0}^d a_\ell f(u_{n-\ell})
\end{equation}
(two-dimensional equations, and equations where the coupling is not unidirectional, are also considered in \cite{hsu:2000}).  The chief motivation for \cite{hsu:2000} lies in the case that $f$ is piecewise linear; in this case \eqref{eq:cnn} becomes a so-called \emph{cellular neural network} (CNN).  Cellular neural networks were first introduced in \cite{chua:1988a} to model the behavior of a large array of coupled electronic components.  

At the cost of modifying $f$ we can assume without loss of generality that $z = 0$.  For simplicity we restrict attention to the case $d = 1$, $z = 0$, and $a_0,a_1 > 0$ so that \eqref{eq:cnn} becomes
\begin{equation}\label{eq:cnn1}
u_n' = -u_n + a_0 f(u_n) + a_1 f(u_{n-1}).
\end{equation}
In this particular case, the conditions (G1) rewrite as follows:
\begin{itemize}
\item ({\bf G1.1}) $f(0) = 0$, $f(1) = \frac{1}{a_0 + a_1}$, and $f(s) \ne \frac{s}{a_0 + a_1}$ for any $s \in (0,1)$;
\item ({\bf G1.2}) $0 \leq a_0 f'(s) < 1$ for $s \in \{0,1\}$, and $f'(1) < \frac{1}{a_0 + a_1} < f'(0)$;
\item ({\bf G1.3}) There is some  $(\beta_0,\beta_1)$ with $\beta_0 < 0$ and $\beta_0 + \beta_1 > 0$ such that $0 < -s_0 + a_0f(s_0) + a_1f(s_1) < \beta_0 s_0 + \beta_1 s_1$ for $0 < s_0 < s_1 < 1$.  
\end{itemize}

The existence problem in \cite{hsu:2000} is formulated in terms of increasing traveling waves of negative speed; a change of variables is required to convert to our setting.  In terms of \eqref{eq:cnn1}, the main hypotheses in \cite{hsu:2000} are
\begin{itemize}
\item $1 < \frac{f'(0)}{a_0 + a_1}$; 
\item $a_1 f(s) \ \mbox{is strictly increasing for} \ s \in (0,1)$.
\item $-s_0 + a_0(f_0) + a_1f(s_0) < \beta_0 s_0 + \beta_1 s_1$ for $\beta_0 = a_0f'(0) - 1$ and $\beta_1 = a_1 f'(0)$.
\end{itemize}
The first and third conditions above are analogous to our (G1.2) and (G1.3).  The second condition above is stronger; it should be thought of as analogous to \eqref{eq:quasimon} and allows for the application of monotone iteration techniques.  Under these conditions, there is some $c_* \ge 0$ such that a decreasing front exists for all $c \geq c_*$ (Theorem 1.1 in \cite{hsu:2000}, reformulated for our setting).  In \cite{hsu:2000} and the companion paper \cite{hsu:1999} the authors also consider \eqref{eq:cnn1} with a piecewise linear $f$ for which the monotonicity condition above fails, but which is simple enough to admit detailed analysis.  In results analogous to ours, the authors describe condititions under which, as $c$ drops below a critical level, monotone fronts give way to non-monotone fronts that ``overshoot" and oscillate about their limit at $\infty$.  

{\flushleft {\bf Clock model for a dilute gas}}

Models of the type \eqref{MODEQ} also arise in situations where the lattice sites represent agents or categories, rather than spatial locations.  One example is considered in \cite{vanzon:1998}, where the authors derive the equation 
\begin{equation}\label{eq:vzex}
w_p' = -w_p + w_{p-1}^2, \ p \in \Z,
\end{equation} 
where each $w_p(t)$ is the proportion of particles in a dilute gas with a particular collision history.   We refer to \cite{vanzon:1998} for the details of the derivation; the model \eqref{eq:vzex} is also discussed in \cite{ebert:2002,peletier:2004,vansaarloos:2003}.   Under the change of variables $u = 1-w$, \eqref{eq:vzex} becomes the equation  we introduced just before Hypothesis \ref{hyp:hyp}: 
\begin{equation}\label{eq:vzexcov}
u'_p = g(u_p,u_{p-1}) = -u_p + 2u_{p-1} - u_{p-1}^2.
\end{equation}
This is a special case of equation \eqref{eq:cnn1} above. 

Similar in spirit  to \cite{hsu:1999} (and a chief inspiration for the present work) is \cite{peletier:2004}, in which Peletier and Rodriguez consider a lattice equation with a piecewise linear nonlinearity that is intended to imitate equation \eqref{eq:vzex}.  In \cite{peletier:2004} the authors describe the behaviors of traveling waves of different speeds: in particular, as the wave speed drops below a certain critical value the fronts change from being monotonically decreasing to being oscillatory.  We will discuss \cite{peletier:2004} further in Section \ref{EXAMPLES}.  

Note that if we write $\eqref{eq:vzexcov}$ in the general form 
\be u'_p = -u_p + f(u_{p-1}) \label{eq:vzgen} \ee the conditions (G1) become
\begin{itemize}
\item (G1.1) $f(0) = 0$, $f(1) = 1$, and $f(s) \ne s$ for $s \in (0,1)$;
\item (G1.2) $f'(0) > 1 > f'(1) \geq 0$;
\item (G1.3) There is some $\beta_1 > 1$ such that $s \leq f(s) < \beta_1 s$ for $s \in (0,1)$. 
\end{itemize}
In this case we can take $\beta_0 = -1$ and $\beta_1 > \left(\sup_{t \in (0,1)} \frac{f(t)}{t} \right)$.  If $f$ is concave down on $(0,1)$ we can take $\beta_1 = f'(0)$.  We revisit \eqref{eq:vzgen} in Section \ref{EXAMPLES}. 

{\flushleft {\bf Semi-discrete advection-reaction equations}}

Genuinely spatially discrete models with unidirectional coupling do arise; given the unidirectionality assumption we impose on \eqref{MODEQ}, though, perhaps the most natural source of applications of \eqref{MODEQ} is the spatial discretization of certain first-order PDEs.  Consider, for example, the PDE 
\be u_t = -f(u)_x + h(u). \label{eq:PDE} \ee  
The so-called \emph{first-order upwind semi-discretization} with step-size $\eps$ for this equation is
\[ 
u'_n = -\frac{1}{\eps}(f(u_n)-f(u_{n-1})) + h(u_n) = g_\eps(u_n,u_{n-1}). 
\]
Since in this case we regard the lattice points as being separated by a distance of $\eps$, the traveling wave ansatz in this case is $u_n(t) = \phi(\eps n - ct)$ and the wave profile equation becomes
\be
\renewcommand{\theequation}{\arabic{equation}$_\epsilon$}
c \phi'(x) = \frac{1}{\eps}[f(\phi(x)) - f(\phi(x-\eps))] - h(\phi(x)). \label{eq:wpecl}
\ee \renewcommand{\theequation}{\arabic{equation}}
In this situation, we can write the conditions (G1) as follows.
\begin{itemize}
\item (G1.1) $h(0) = h(1) = 0$, and $h(s) \ne 0$ whenever $s \in (0,1)$;
\item (G1.2) $h'(1) < 0 < h'(0)$, $f'(1) \ge 0$, and $\frac{1}{\epsilon} f'(0) > h'(0)$;
\item (G1.3) $\beta_0 < 1$, $\beta_0 + \beta_1 > 0$, and $0 < \eps h(s_0) + f(s_1) - f(s_0) < \eps \beta_0 s_0 + \eps \beta_1 s_1$ whenever $0 < s_0 < s_1 < 1$.  
\end{itemize}
Observe that, for these conditions to hold, the source term $h$ must be {\it monostable} --- $h(s) > 0$ for $s \in (0,1)$ with $h'(0) > 0 > h'(1)$; thus semi-discrete conservation laws ($h \equiv 0$), which are studied in \cite{benzoni-gavage:1998,bianchini:2003a,serre:2007}, are excluded from our analysis.  Observe also that if $f$ is decreasing on any subinterval of $[0,1]$, then the first inequality in (G1.3) can be violated by choosing $\eps$ sufficiently small.  However, so long as the flux $f$ is non-decreasing on $[0,1]$, and the source term $h$ is monostable, the conditions can be satisfied for $\eps > 0$ as small as we wish.

Let us assume that (G1.1) and (G1.2) hold and that $f$ is smooth and nondecreasing on $[0,1]$.  Put $c_* = \sup_{s \in (0,1)} f'(s) > 0$ and $h_* = \sup_{s \in (0,1)} h(s)/s$.  Then, given any $\delta > 0$, for $\epsilon > 0$ sufficiently small the choice
\[
\beta_1 = \frac{c_* + \delta}{\epsilon} \ \mbox{and} \ \beta_0 = h_* - \beta_1
\]   
makes (G1.3) hold.  For such $\beta = (\beta_0,\beta_1)$, 
\[
c(\beta) = \inf_{\lambda > 0} \frac{\epsilon h_* + (e^{\epsilon \lambda} - 1)(c_* + \delta)}{\epsilon \lambda}.
\]
As $\epsilon \to 0$, $c(\beta) \to c_* + \delta$.  Thus, for any $c > c_*$, $(\ref{eq:wpecl}_\epsilon)$ admits a unique decreasing front $\phi_\epsilon$ for $\epsilon$ sufficiently small.  For such a $\phi_\epsilon$, we have $|\phi_\epsilon'| \leq |h(\phi_\epsilon)|/(c - c_*)$ everywhere.  Thus the $\phi_\epsilon$ form a bounded equicontinuous family as $\epsilon \to 0$, and so converge uniformly on compact sets to a solution of the continuum limit of $(\ref{eq:wpecl}_\epsilon)$, 
\be 
\phi' = \frac{-h(\phi)}{c-f'(\phi)}. 
\label{eq:wpecl_cont} 
\ee  
$\eqref{eq:wpecl_cont}$ is a first order scalar ODE, and so can be analyzed easily: since $h$ is positive on $(0,1)$, classical solutions exist that connect $1$ to $0$ if and only if $c > c_*$.  In other words, the monotone fronts obtained in Theorem $\ref{thm:main}$ for $(\ref{eq:wpecl}_\epsilon)$ with $c > c(\beta(\eps))$ are lattice realizations of the continuous monotone fronts of $\eqref{eq:wpecl_cont}$.

Note that in the case that $c_m < c_*$, the solutions $\phi_\eps$ to $(\ref{eq:wpecl}_\epsilon)$ with $c \in [c_m,c_*)$ converge pointwise almost everywhere, via Helly's Theorem, to a monotone function $\phi_*$ with countably many jump discontinuities.  Furthermore $\phi_*$ satisfies $\eqref{eq:wpecl_cont}$ on its intervals of continuity and satisfies the boundary conditions $\phi_*(-\infty) = 1$ and $\phi_*(\infty) = 0$.  In other words, $\phi_*$ is an entropy solution of $\eqref{eq:wpecl_cont}$.  In \cite{Mascia:2000} it is shown that entropy solutions of $\eqref{eq:wpecl_cont}$ exist if and only if $\sup_{s \in (0,1)} \frac{f(s)}{s} \le c \le c_*$, suggesting that $\lim_{\eps \to 0} c_m(\eps) = \sup_{s \in (0,1)} \frac{f(s)}{s}$.  This analysis highlights the fact that the estimate $c(\nabla g({\bf 0})) \le c_m \le c(\beta)$ does not determine $c_m$ exactly.  The determination of $c_m$ in any particular case with $c(\nabla g({\bf 0})) < c(\beta)$ requires a careful analysis of the particular problem under study.

In addition to the continuum limit $\eps \to 0$, one can look at the conservation law limit $h \to 0$ for fixed $\eps$.  In this limit it is not hard to show that the corresponding wave profiles $\phi^h$ converge, uniformly on compact sets along a subsequence, to constant functions.  Thus the fronts that we study can be thought of as bifurcating from equilibria of the conservation law $u'_n = \frac{1}{\eps}(f(u_n)-f(u_{n-1}))$.

 {\flushleft {\bf Two coupling terms}}

Two more equations similar to $\eqref{eq:cnn}$ are the two-dimensional equation 
\[ 
u'_{i,j} = -u_{i,j} + \alpha f(u_{i,j-1}) + \beta f(u_{i-1,j}) 
\]
and the one-dimensional equation
\[
u'_i = -u_i + \alpha f(u_{i-1}) + \beta f(u_{i-2}),
\]
where $\alpha$ and $\beta$ are positive parameters.  In each of these cases $g$ is a function from $\R^3$ to $\R$ --- though, in the two-dimensional equation, if the wave direction is $\xi = (1,1)$ then the two delayed arguments to $g$ in $\eqref{eq:wpe}$ collapse to a single argument.   We close this section with a remark about the wave direction $\xi$ in the two-dimensional equation.  With $g$ and the coupling function $\chi$ viewed as given, the set of wave directions $\xi$ that satisfy $\eqref{eq:unidir}$ is closed and connected in $S^2$.Ê Theorem $\ref{thm:main}$ applies to each of these wave directions, with the quantity $c(\beta)$ now regarded as a function of $\xi$ --- or, equivalently, as a function of $\kappa$.Ê Suppose that $\beta$ satisfies our hypotheses and let $F(x;\beta,\kappa)=\frac{\beta_0+\sum_{i=1}^N \beta_i e^{\kappa_i x}}{x}$.  As we shall see below, $F$ has a unique positive critical point and a corresponding minimum value $c(\beta) = c(\beta;\kappa) > 0$.  Since $F$ is smoothly increasing in each $\kappa_i$ (for positive $x$), so is $c(\beta;\kappa)$.    As mentioned above, in the $\beta = \nabla g({\bf 0})$ case $c(\beta)$ is the asymptotic spreading speed for the linearization of \eqref{MODEQ} about zero.  In this light the monotonicity of $c(\beta)$ in $\kappa$ can be interpreted to mean that $\eqref{MODEQ}$ transports mass more quickly in directions for which $\mathbf{\chi}_i(p)$ lies further behind the hyperplane orthogonal to $\xi$.Ê The fact that $c$ is smooth in $\kappa$ can be contrasted with \cite{cahn:1999, hoffman:2009} where a parameter that controls whether or not fronts are present is discontinuous with respect to the direction in which the front travels.

\subsection*{Acknowledgements}
We thank L. A. Peletier for sharing his work \cite{peletier:2004}, for suggesting that extending it would make an interesting project, and for useful discussions at an early stage of this project.  We thank A. Scheel for helpful discussions and in particular for bringing the work of Van Saarloos \cite{vansaarloos:2003} to our attention.  We thank C.E. Wayne for making helpful remarks on an early version of this paper.  Finally, we thank the anonymous referees for useful suggestions.
\noindent
This work was funded in part by the National Science Foundation under grant DMS-0603589.

\section{Linearization of \eqref{eq:wpe} and the unstable manifold at $1$}

In this section we present some basic facts about the characteristic equation $D(\lambda;c,\alpha) = 0$, where $D$ is as in \eqref{eq:Ddef}, and present some results regarding the unstable manifold for \eqref{eq:wpe} at $1$.  Much of this material is well-known. 

\begin{proposition}[Description of the roots of the characteristic equation]
\label{lem:cdef}
Let $\alpha_0 < 0 \le \alpha_i, \ i =1, \ldots, N$ and $\kappa_i > 0, \ i = 1,\ldots, N$ be given.  Take $c > 0$.  
We regard 
\begin{equation}\label{eq:Ddef2}
D(\lambda;c,\alpha) := c\lambda + \alpha_0 + \sum_{i=1}^N \alpha_i e^{-\kappa_i \lambda}
\end{equation}
as a function of $\lambda \in \mathbb{C}$.  We define
\begin{equation*}
c(\alpha) := \inf_{\lambda > 0} \frac{\alpha_0 + \sum_{i=1}^N \alpha_i e^{\kappa_i \lambda}}{\lambda}
\end{equation*}
(recall \eqref{eq:cdef}).  

The following hold.

\begin{enumerate}
\item The real part of any complex root of $D$ is less than any real root of $D$. 
\item If $ \sum_{i=0}^N \alpha_i \leq 0$, then $D$ has one nonnegative real root and one nonpositive real root, and all complex roots have negative real parts.   If $\sum_{i=0}^N \alpha_i < 0$, the two real roots are in fact positive and negative.   
\item If $\sum_{0=1}^N \alpha_i > 0$, then $c(\alpha)$ is finite and $0 < b(\alpha) < \sum_{i=1}^N \alpha_i \kappa_i < c(\alpha)$.  
\item If $\sum_{0=1}^N \alpha_i > 0$, then $D$ has real roots if and only if $c \geq c(\alpha)$.  
\end{enumerate}
\end{proposition}
\begin{remark}
When $\alpha_0 = -1$, $\alpha_1 > 1$, and $N = 1$, a complete description of the roots of $D$ is given in Theorem 6.1 in \cite{mallet-paret:1988}.
\end{remark}
\begin{proof}[Proof of Lemma $\ref{lem:cdef}$]
We begin with the case that $\alpha_i = 0$ for all $i \in \{1,\ldots,N\}$; in this case we need to prove the first two points of the proposition (since $\sum_{0=1}^N \alpha_i > 0$ implies that $\alpha_i > 0$ for some positive $i$).  In this case, $D(\lambda) = c\lambda + \alpha_0$ has one root, this root is positive, and the proposition holds.  We henceforth assume that $\alpha_i > 0$ for at least one $i \in \{1,\ldots,N\}$.

Consider the restriction 
\[
D(x;c,\alpha) = cx + \alpha_0 +  \sum_{i=1}^N \alpha_i e^{-\kappa_i x}
\]
of $D$ to the real axis.  Let $D'(x;c,\alpha) = c - \sum_{i=1}^N \alpha_i \kappa_i e^{-\kappa_i x}$ denote the derivative of $D$ with respect to $x$.
Note that $D''(x;c,\alpha) = \sum_{i=1}^N \alpha_i \kappa_i^2 e^{-\kappa_i x} > 0$, that $D(-\infty) = D(\infty) = \infty$, and that $D'(-\infty) = -\infty$ and $D'(\infty) = c > 0$.  Thus $D$ decreases to a single global minimum $\tilde{D} = \tilde{D}(c,\alpha)$ at some $\tilde{x} = \tilde{x}(c,\alpha)$, and thereafter increases.  

We now prove part (1) of the proposition.  Assume that $\lambda_0 = x_0 + iy_0$ is a complex root of $D$ (with $y_0 \neq 0$).  Writing the real and imaginary parts of the equation $D = 0$ separately yields
\[ 
\ba{l} \alpha_0 + cx_0 + \sum_{i=1}^N \alpha_i e^{-\kappa_i x_0} \cos(\kappa_i y_0) = 0; \\ \\ 
cy_0 + \sum_{i=1}^N \alpha_i e^{-\kappa_i x_0} \sin(\kappa_i y_0) = 0.
\ea 
\]
If the above equations hold, we see that
\[
D(x_0;c,\alpha) = \alpha_0 + cx_0 + \sum_{i=1}^N \alpha_i e^{-\kappa_i x_0} \cos(0) \geq 0,
\]
with strict inequality unless $\cos(\kappa_i y_0) = 1$ for all $i \in \{1,\ldots,N\}$.  The second equation above rewrites as
\[
c = -\sum_{i=1}^N \alpha_i e^{-\kappa_i x_0} \frac{\sin(\kappa_i y_0)}{y_0}.
\]
Thus at $x_0$ we have
\[
\frac{d}{dx}D(x_0;c,\alpha) = c - \sum_{i=1}^N \alpha_i \kappa_i e^{-\kappa_i x_0} = -\sum_{i=1}^N \alpha_i\kappa_i e^{-\kappa_i x_0} \left( \frac{\sin(\kappa_i y_0)}{\kappa_i y_0} + 1 \right) < 0.
\]
Thus, viewing $D$ as a function of the real variable $x$, we see that $D$ is nonnegative and decreasing at $x_0$.  Thus $x_0$ is less than or equal to any real root of $D$.   If $x_0$ is in fact a real root of $D$, then $\cos(\kappa_i y_0) = 1$ for all $i \in \{1,\ldots,N\}$ and the equation $D = 0$ must hold both at $x_0 + iy_0$ and at $x_0$.  The imaginary part of $D = 0$ shows that this is impossible; we conclude that $x_0$ lies strictly to the left of any real roots of $D$.  This completes the proof of part (1) of the proposition.

Since
\[
D(0;c,\alpha) = \sum_{i=0}^N \alpha_i,
\]
if $\sum_{i=0}^N \alpha_i \leq 0$ we see that $D$ has a nonnegative real root and a (not necessarily distinct) nonpositive real root.  If $\sum_{i=0}^N \alpha_i < 0$, $D$ has a positive root and a negative root.  Part (1) now implies that all nonreal roots necessarily have negative real parts.  We have proven part (2) of the proposition.  

We now assume that $\sum_{i=0}^N \alpha_i > 0$.  With this assumption, let us define the function 
\begin{equation}\label{eq:Fdef}
F(x;\alpha) = F(x) := \frac{1}{x}\left[\alpha_0 + \sum_{i=1}^N \alpha_i e^{\kappa_i x} \right]
\end{equation}
for $x > 0$.  Note that $F(0+) = F(\infty) = \infty$ and that $F$ is continuous; thus $F$ is minimized at some positive point and its minimum value is $c(\alpha)$.   

We compute
\[
F(x) = \frac{1}{x}\left[ \sum_{i=1}^N \alpha_i e^{\kappa_i x} + \alpha_0 \right] > \frac{1}{x} \left[ \sum_{i=1}^N \alpha_i (1 + \kappa_i x) +\alpha_0 \right] = \sum_{i=1}^N \alpha_i \kappa_i + \frac{\sum_{i=1}^N \alpha_i + \alpha_0}{x} > \sum_{i=1}^N \alpha_i \kappa_i.
\]
Taking the minimum over $x > 0$ of the left hand side yields $c(\alpha) > \sum_{i=1}^N \alpha_i \kappa_i$.  To prove part (3) of the proposition, it remains to show that $b(\alpha) < \sum_{i=1}^N \alpha_i \kappa_i$; that is, for $c \ge \sum_{i=1}^N \alpha_i \kappa_i$, all roots have negative real part.

Let us assume accordingly that $c \ge \sum_{i=1}^N \alpha_i \kappa_i$.  In this case a simple computation shows that $D(0) > 0$ and $D'(0) \ge 0$, so that any real roots of $D$ are strictly negative.  Thus, if $D$ has real roots, the complex roots must have negative real parts by part (1) of the lemma.  Even if no real roots are present any complex root must have negative real part: for given a root $x_0 + iy_0$ with $y_0 \neq 0$, we know from above that $D'(x_0;c,\alpha) < 0$, whence $x_0$ is less than $0$. 

It remains to prove part (4) of the proposition: still assuming that $\sum_{i=0}^N \alpha_i > 0$, we wish to show that $D$ has real roots precisely when $c \geq c(\alpha)$.  
Let us write $\tilde{D}$ for the minimum of $D$ over the reals and $\tilde{x}$ for the point at which this minimum is attained (our earlier observations imply that $\tilde{x}$ is negative).  We have
\begin{equation}
 c = \sum_{i=1}^N \alpha_i \kappa_i e^{-\kappa_i \tilde{x}},
\label{eq:csgn}
\end{equation}
which we substitute into $\eqref{eq:Ddef2}$ to obtain
\begin{equation}
\tilde{D}(\tilde{x}) = \alpha_0 + \sum_{i=1}^N \alpha_i (1 + \tilde{x} \kappa_i) e^{-\kappa_i \tilde{x}}. 
\label{eq:Dtilde}
\end{equation}
Note that
\[ 
\tilde{D}'(\tilde{x}) = -\tilde{x}\sum_{i=1}^N \alpha_i \kappa_i^2 e^{-\kappa_i \tilde{x}} > 0. 
\]
Since $\tilde{x}$ is defined implicitly in terms of $c$, we may regard $\eqref{eq:Dtilde}$ as defining $\tilde{D}$ as a function of $c$ (holding $\alpha$ fixed). 
Since $\tilde{x}$ is strictly decreasing in $c$ and $\tilde{D}$ is strictly increasing in $\tilde{x}$ for $\tilde{x} < 0$, it follows that $\tilde{D}$ is strictly decreasing in $c$.   As $c$ approaches $\sum_{i=1}^N \alpha_i \kappa_i$ from above, $\tilde{x}$ approaches $0$ from below and $\tilde{D}$ approaches $\sum_{i=0}^N \alpha_i > 0$ from below.  On the other hand, as $c \to \infty$, $\tilde{x}$ approaches $-\infty$ and $\tilde{D}$ approaches $-\infty$.  Therefore there is some unique $c_*$ where $\tilde{D}$ --- viewed as a function of $c$ --- crosses zero downward as $c$ crosses $c_*$ upward.  When $\tilde{D} \le 0$, $D$ has real roots.  

Our claim is that $c_* = c(\alpha)$.  To establish this we compute 
\[
F'(x) = \frac{1}{x^2}\left[ -\alpha_0 + \sum_{i=1}^N \alpha_i (\kappa_i x - 1)e^{\kappa_i x}\right]
\]
and observe that 
\[
F'(-\tilde{x}(c)) = \frac{-\tilde{D}(c)}{\tilde{x}(c)^2}.
\]
Since $c_*$ is the unique positive root of $\tilde{D}$, it follows that $-\tilde{x}(c_*)$ is the unique positive critical point of $F$ --- that is, the global minimum of $F$.  Thus we have 
\[
c(\alpha) = F(-\tilde{x}(c_*)).  
\]
To see in turn that $c_* = F(-\tilde{x}(c_*))$, we compute
\begin{eqnarray*}
0 & = & \tilde{D}(c_*) \ \iff \\
-\alpha_0 & = & \sum_{i=1}^N \alpha_i (\tilde{x}(c_*) \kappa_i +1)e^{-\kappa_i \tilde{x}(c_*)} \ \iff \\
\frac{\sum_{i=1}^N \alpha_i e^{\kappa_i (-\tilde{x}(c_*))} + \alpha_0}{-\tilde{x}(c_*)} & = & \sum_{i=1}^N \alpha_i \kappa_i e^{-\kappa_i \tilde{x}(c_*)} \ \iff \\
F(-\tilde{x}(c_*)) & = & c_*,
\end{eqnarray*}
where the last line follows from \eqref{eq:csgn} and from the definition of $F$.  
\end{proof}

The above proposition establishes all of point 1 of Theorem \ref{thm:main} except that $c(\beta) \geq c(\nabla g({\bf 0}))$.  Our hypotheses on $\beta$ imply that $\beta_i \geq \del g ({\bf 0})/\del s_i$ for all $i \in \{1,\ldots,N\}$, and that $\sum_{i=0}^N \beta_i \geq \sum_{i=0}^N \del g ({\bf 0})/\del s_i$.  It follows that $F(x;\beta) \geq F(x;\nabla g({\bf 0}))$ for all $x > 0$ (see \eqref{eq:Fdef}); the desired inequality follows in turn.

Our hypotheses on $g$ and the above proposition yield that, for every $c > 0$, the characteristic equation for the linearization of \eqref{eq:wpe} at $1$ has exactly one positive root $\lambda$, with all other roots having negative real parts.  Accordingly, equation \eqref{eq:wpe} has a one-dimensional unstable manifold at $1$.  We now review some facts about one-dimensional unstable manifolds at a hyberbolic equilibrium point; we refer the reader to chapters 7 and 10 of \cite{hale:1993}.  Recall that we are writing $C = C[-r,0]$ for the phase space for \eqref{eq:wpe}.  Following \cite{hale:1993} let us write 
\[
W^u(1) = \{ \ y_0 \in C \ : \ y(t) \ \mbox{solves \eqref{eq:wpe} for all $t < 0$ and $y(-\infty) = 1$} \ \}.
\]
Given a neighborhood $V$ of the constant function $1$ in $C$, we write 
\[
W^u(1,V) = \{ \ y_0 \in W^u(1) \ : \ y_t \in V \ \mbox{for all} \ t \leq 0 \ \}. 
\]
The following proposition is just Theorem 1.1 in chapter 10 of \cite{hale:1993} (along with some facts that emerge in the proof), adapted to our situation.  
\begin{proposition}\label{prop:umfact}
Let $c > 0$ and let $\lambda$ denote the unique positive real root of the characteristic equation for $\eqref{eq:wpe}$ at the equilibrium $1$. 
There is a neighborhood $V$ about $1$ in $C$ and a linear projection 
\[
\pi : C \to U := \{ ke^{\lambda \cdot}, \ k \in \R \} \subset C
\]
such that the following holds.  There is some neighborhood $\tilde{V}$ of $0$ in $U$ and a $C^1$ function $h: \tilde{V} \to C$ such that $W^u(1,V) = 1 + h(\tilde{V})$, $h(0) = 0$, and such that $h$ satisfies the estimate 
\[
\|h(\phi) -  h(\psi)\| \geq \|\phi - \psi\|/2.
\]
\end{proposition}

We apply these facts in the following lemma. 

\begin{lemma} \label{lem:W^u}
Given any $c > 0$, up to time translation there are exactly two nontrivial solutions $\phi_+$ and $\phi_-$ of
$\eqref{eq:wpe}$ which satisfy the boundary condition $\eqref{eq:negbc}$.  There is some $x_0$ such that $\phi_-(x)-1$ is strictly negative and $\phi_+(x)-1$ is strictly positive for all $x \leq x_0$.  
\end{lemma}

\begin{proof}
Let $V$ be as in the above proposition.  Any solution $\phi$ of \eqref{eq:wpe} satisfying \eqref{eq:negbc} certainly satisfies $\phi_x \in W^u(1,V)$ for all sufficiently negative $x$.  For such $x$, let us write $\phi_x - 1 = y_x = h(ke^{\lambda \cdot})$, where $h$ is as in Proposition \ref{prop:umfact}.   

Since $h$ is $C^1$, we have 
\[
y_x = \gamma k e^{\lambda \cdot} + f(k e^{\lambda \cdot}), 
\]
where $\|f(k e^{\lambda \cdot})\| \leq \epsilon \|k e^{\lambda \cdot}\|$ for any given $\epsilon$, provided that $k$ is close enough to zero.  The expansivity condition in Proposition $\ref{prop:umfact}$ implies that $\gamma$ is nonzero.  

Observe that the point $k e^{\lambda \cdot} \in C$ has maximum absolute value $|k|$ and minimum absolute value $|k|e^{-\lambda r}$.  Therefore, by choosing $k$ close enough to $0$ we have that $\|f(k e^{\lambda x})\| < |\gamma k| e^{-\lambda r}$, whence we can conclude that $y_x$ is all of one sign and of the same sign as $\gamma k e^{\lambda \cdot}$.  Thus, for any $\phi \in W^u(1,V)$ sufficiently close to $1$, $\phi - 1$ is all of one sign, and both signs are possible.  

Thus $\phi_-$ and $\phi_+$ as described in the lemma exist.  The uniqueness of $\phi_-$ and $\phi_+$ up to time translation can be proven by noting that $h$ is a diffeomorphism near zero; we omit the details.   
\end{proof}

\begin{proposition}\label{prop:uccs}
Suppose that $c_k \to c_0 \neq 0$ as $k \to \infty$, and that a globally defined solution $\phi_k : \R \to \R$ of \eqref{eq:wpe} of speed $c_k$ exists for all $k$.  Suppose moreover that there is some $M$ such that $|\phi_k(x)| \leq M$ for all $k \in \N$ and all $x \in \R$.  Then $\phi_k$ converges, uniformly on compact subsets of $\R$, to a solution $\phi_0$ of \eqref{eq:wpe} of speed $c_0$.  
\end{proposition}

We emphasize that, even if the $\phi_k$ all connect $1$ to $0$,  $\phi_0$ need not do so. 

\begin{proof}
Since the bounds $|\phi_k(x)| \le M$ and 
\[
|\phi'(x)| \le \frac{\sup\{ \ |g(s)| , \ |s| \leq M \  \}}{\min_k c_k}
\]
hold uniformly for $x \in \R$ and $k \in \mathbb{N}$, the sequence $\phi_k$ is uniformly equicontinuous and pointwise bounded and hence, by Ascoli's Theorem and a diagonalization argument, has a subsequence that converges on compact sets to some limit $\phi_0$.  Since the $\phi_k$ satisfy $\eqref{eq:wpe}$, $\phi_k'$ is equicontinuous and pointwise bounded as well and hence also converges uniformly on compact sets to some limit.  This limit is equal to $-g(\Phi_0(x))/c$ for all $x$, but is also equal to $\phi_0'$ by a standard theorem (see, for example, Theorem 7.17 of \cite{rudin}).  It follows that $\phi_0$ is a solution of \eqref{eq:wpe}.
\end{proof}

We close this section by establishing an estimate on the basin of attraction for \eqref{eq:wpe} about $0$ that is locally uniform in $c$.    Suppose that $T_c(t): C \to C$ is the solution operator for \eqref{eq:wpe} (with $c$ the wave speed) and that $L_c(t): C \to C$ is the solution operator for the linearization of \eqref{eq:wpe} at zero:
\begin{equation}\label{eq:lwpe}
cy'(t) = - \nabla g({\bf 0}) \cdot Y(t), \ \ Y(t) = (y(t),y(t-\kappa_1),\ldots,y(t - \kappa_N)). 
\end{equation}
As proved above, as long as $c > b(\nabla g({\bf 0}))$ there is some $\lambda_c > 0$ such that every root of the characteristic equation at zero has real part less than $-\lambda_c$.  For every such $c$ and $\lambda_c$, it is standard that there is a $K_c > 0$ such that 
\[
\|L_c(t) y_0\| \leq K_c e^{-\lambda_c t} \|y_0\|.
\]

We will need the following elementary lemma. 

\begin{lemma} \label{lem:linunifc}
Given $c_1 > b(\nabla g({\bf 0}))$, $\epsilon > 0$, and $\tau > 0$, there is a $\delta > 0$ such that $|c - c_1| \leq \delta$ implies that 
\[
\|L_{c}(t)y_0 - L_{c_1}(t)y_0\| \leq \epsilon\|y_0\| 
\]
for all $t \in [0,\tau]$.  
\end{lemma}

\begin{proof}
With notation as in the statement of the lemma, write $c = c_1 + \eta$.  Given $y_0 \in C$, let us write $y$ for the continuation of $y_0$ as a solution of \eqref{eq:lwpe} with wave speed $c_1$ and $w$ for the continuation of $y_0$ as a solution of \eqref{eq:lwpe} with wave speed $c$ --- so 
\[
y_t = L_{c_1}(t)y_0 \ \ \mbox{and} \ \ w_t = L_c(t)y_0.
\]
Write $A y_t = - \nabla g({\bf 0}) \cdot Y(t)$; we have $y'(t) = Ay_t/c_1$ and $w'(t) = Aw_t/c$.  
For all $t>0$ we have the bounds
\[
|y'(t)| \leq \frac{M}{c_1}\|y_t\| \ \mbox{and} \ |w'(t)| \leq \frac{M}{c} \|w_t\|,
\]
where $M > |\nabla g({\bf 0})|$.  

Choose $d_1 > 0$ such that $\|y_t\| \leq d_1 \|y_0\|$ for all $t \in [0,\tau]$.  

Now, observe that $\|w_\tau - y_\tau\| \leq \int_0^\tau \|w'(t) - y'(t)\| \ dt$, and that for all $t \in [0,\tau]$ we have   
\begin{eqnarray*}
\|w'(t) - y'(t)\| & = & \left\| \frac{Aw_t}{c} - \frac{Ay_t}{c_1} \right\| \\
& = & \left\| \frac{A(w_t - y_t)}{c} + \frac{Ay_t}{c} - \frac{Ay_t}{c_1} \right\| \\
& \leq & \frac{M}{c}\|w_t - y_t\| + \frac{M\eta }{c_1(c_1+\eta)}\|y_t\| \\
& \leq & \frac{M}{c}\|w_t - y_t\| + \frac{d_1 M\eta }{c_1(c_1+\eta)}\|y_0\|.
\end{eqnarray*}
Comparing $\|w_t - y_t\|$ to the solution of the ODE
\[
u' = \frac{M}{c} u + \frac{d_1 M\eta }{c_1(c_1+\eta)}\|y_0\|, \ u(0) = 0
\]
yields the desired result for $|\eta|$ sufficiently small.  
\end{proof}

We now make use of two standard facts.  First, $L_c(t)$ is the derivative of $T_c(t)$ with respect to its functional coordinate at $0$.   Second, $T_c(t)$ is uniformly differentiable at $0$ with respect to $t$ and $c$, so we have the following: given any $\eta > 0$, $\tau_0 > 0$, $c_1 >  b(\nabla g (\bf{0}))$, and $\delta \in (0,c_1 - b(\nabla g (\bf{0})))$, there is some $\epsilon > 0$ such that $\|y_0\| \leq \epsilon$ implies that 
\[
\|T_c(t)y_0 - L_c(t)y_0\| \leq \eta \|y_0\| \ \mbox{for all} \ t \in [0,\tau_0] \ \mbox{and} \ c \in [c_1 - \delta,c_1 + \delta].  
\] 
(See, for example, sections VII.5 and VII.6 of \cite{diekmann:1995}.)  The following proposition can be obtained from Theorem VII.1.3 in \cite{daleckii:1974} together with the fact that the sun-star calculus (see \cite{diekmann:1995}) can be used to write $\eqref{eq:wpe}$ as an ODE in a Banach space.  For completeness, we include the proof here.  The ideas are similar to those in section VII.5 of \cite{diekmann:1995}.  

\begin{proposition} \label{prop:unif_basin}
Let $c_1 > b(\nabla g({\bf 0}))$ be given.  Then there is a $\delta$ small enough that the following holds.  There is an $\eps > 0$ such that whenever $|c-c_1| < \delta$, then the basin of attraction for $\eqref{eq:wpe}$ about $0$ includes the ball of radius $\eps$ centered at the origin.
\label{prop:zero_stable}
\end{proposition}
\begin{proof}
As above, choose $K>0$ and $\lambda>0$ such that 
\[
\|L_{c_1}(t)y_0\| < Ke^{-\lambda t} \|y_0\|
\] 
for all $t \geq 0$.  

Choose $\tau$ large enough that $K e^{-\lambda \tau} \leq 1/8$.  Now choose $\delta$, as in Lemma $\ref{lem:linunifc}$, small enough so that 
\be
\|L_{c_1}(t)y_0 - L_c(t)y_0\| \leq \|y_0\|/8 \ \mbox{for all}  \ t \in [0,\tau] \label{eq:propest1}
\ee
whenever $|c - c_1| < \delta$.  Now choose $\epsilon > 0$ small enough that $\|y_0\| \leq \epsilon$ implies that, for all $t \in [0,\tau]$ and any $c \in [c_1-\delta,c_1+\delta]$, 
\be
\|T_c(t) (y_0) - L_c(t) y_0\| \leq \frac{\|y_0\|}{8}. \label{eq:propest2}
\ee
At $t = \tau$, then, for $\|y_0\| \leq \epsilon$ we have
\begin{eqnarray*}
& & \|T_c(\tau)(y_0)\| \leq \|T_c(\tau)(y_0) - L_{c_1}(\tau)(y_0)\| + \|L_{c_1}(\tau)(y_0)\| \\
& \leq & \|T_c(\tau)(y_0) - L_{c}(\tau)y_0\| + \|L_{c}(\tau)y_0 - L_{c_1}(\tau)y_0\| +  Ke^{-\lambda t} \|y_0\| \\
& \leq & \|y_0\|/2.  
\end{eqnarray*}
Let $t \ge 0$ be given and write $t = n\tau + t_1$ where $t_1 \in [0,\tau)$.
Write $T_c(t)y_0 = \left(T_c(t_1) - L_c(t_1) + L_c(t_1)\right) T_c(\tau)^n (y_0)$.  Since $\| T_c(\tau)^n y_0 \| \le 2^{-n} \|y_0\| < \eps$ the estimates $\eqref{eq:propest1}$ and $\eqref{eq:propest2}$ hold and we have
\[ \| T_c(t)y_0\| \le (\frac{1}{8} + K)2^{-n} \|y_0\|, \]
which goes to zero as $t \to \infty$.
\end{proof}

\section{Existence and uniqueness of monotone fronts when $c \ge c(\beta)$}

In this section we establish the existence and uniqueness of monotone fronts when $c \geq c(\beta)$.  In particular, we prove statements 2 and 3 of Theorem $\ref{thm:main}$.  Already from last section, we see that the only candidates for a front (i.e. a traveling wave solution connecting $1$ and $0$) are the solutions $\phi_+$ and $\phi_-$ described in Lemma \ref{lem:W^u}; of these, only $\phi_-$ might be an initially decreasing front.  Indeed, henceforth we will be preoccupied almost entirely with $\phi_-$ and whether it satisfies the boundary condition $\eqref{eq:posbc}$.   We begin with the $c > c(\beta)$ case.  Our proof that $\phi_-$ is actually a decreasing front when $c > c(\beta)$ proceeds in two steps: we show that $\phi_-$ is strictly decreasing until it crosses zero; and we show that $\phi_-$ never does cross zero.  We assume that (G1) holds throughout.  

\begin{lemma} \label{lem:decrease}
Let $\phi$ be a solution of \eqref{eq:wpe} that satisfies the boundary condition $\eqref{eq:negbc}$, and suppose that $x_0 \in \R$ is such that $\phi(x) \in (0,1)$ for all $x \leq x_0$.  Then $\phi'(x) < 0$ for all $x \leq x_0$.  In particular, if $\phi(x) \in (0,1)$ for all $x$, then $\phi$ is strictly decreasing everywhere.
\end{lemma}

\begin{proof}
It is sufficient to prove that $\phi'(x_0) < 0$.  Imagine, to the contrary, that $\phi'(x_0) \geq 0$.  Write $d$ for the minimum nonzero delay $\kappa_i$.  We claim that there is some $x_1 \le x_0 - d$ with $\phi'(x_1) \geq 0$ and $\phi(x_1) \leq \phi(x_0)$. Induction then contradicts the assumption that $\phi(-\infty) = 1$.   

We prove the claim.  To say that $\phi'(x_0) \geq 0$ is to say that $g(\Phi(x_0)) \le 0$.   Since $\phi(x_0 - \kappa_i) \in (0,1)$ for each $i$ it follows from the first inequality in (G1.3) that there is some $i$ with $\phi(x_0 - \kappa_i) \leq  \phi(x_0)$.  If $\phi'(x_0 - \kappa_i) \ge 0$, we're done: just take $x_1 = x_0 - \kappa_i$.  If $\phi'(x_0 - \kappa_i) < 0$, then there must be a minimum $y$ on $(x_0 - \kappa_i,x_0)$ such that $\phi'(y) = 0$.  The same argument we just made shows that $\phi(y - \kappa_j) \leq \phi(y)$ for some $j$; since $\phi$ is strictly decreasing on $(x_0 - \kappa_i,y)$ we see that $y - \kappa_j < x_0 - \kappa_i$, and that $\phi(y - \kappa_j) < \phi(x_0 - \kappa_i)$.  By the Mean Value Theorem, $\phi'(x_1) \geq 0$ for some $x_1 \in (y - \kappa_j, x_0 - \kappa_i)$.  This proves the claim and completes the proof.  
\end{proof}

\begin{proposition} \label{prop:exist}
If $c > c(\beta)$, the unique initially decreasing solution $\phi$ of $\eqref{eq:wpe}$ satisfying the boundary condition $\eqref{eq:negbc}$ is strictly decreasing for all time and satisfies the limit $\eqref{eq:posbc}$.  In particular, it is a monotone front.
\end{proposition}

\begin{proof}
Let 
\[
A = \arg\min_{x > 0} \frac{\beta_0 + \sum_{i=1}^N \beta_i e^{\kappa_i x}}{x}.
\]
(The minimum exists because $\sum_{i=0}^N \beta_i > 0$ --- recall Proposition \ref{lem:cdef}.)  Since $c > c(\beta)$, for some $\delta \in (0,A)$ we have
\[
c = \frac{\beta_0 + \sum_i \beta_{i=1}^N e^{\kappa_i A}}{A - \delta}.
\]

Consider the supremum
\[ 
x_0 := \sup\{ x \; | \; \phi'(y) \ge -A\phi(y) \mbox{ for all } y \le x\}.
\]
Since $\phi(-\infty) = 1$ while $\phi'(-\infty) = 0$, the set over which we are taking the supremum is not empty.  We will show that $x_0 = \infty$, precluding that $\phi$ is ever equal to $0$.  

By way of contradiction, imagine that $x_0 < \infty$.  Since $\phi$ is $C^1$, we must in fact have $\phi'(x_0) = -A\phi(x_0)$.  Comparing $\phi$ to the ODE
\[
\psi' = -A\psi, \ \ \psi(x_0) = \phi(x_0)
\]
yields that, for all $\kappa_i$, 
\[
\phi(x_0-\kappa_i) \leq \phi(x_0)e^{A \kappa_i}  
\]
and therefore that
\begin{eqnarray*}
& & \phi'(x_0) = \frac{- g(\phi(x_0),\phi(x_0 - \kappa_1),\ldots,\phi(x_0 - \kappa_N))}{c} \\
& \geq & \frac{-\beta_0\phi(x_0) - \sum_{i=1}^N \beta_i \phi(x_0 - \kappa_i)}{c} \\
& \geq & -\left[\frac{\beta_0 + \sum_{i=1}^N \beta_i e^{A \kappa_i}}{c}\right]\phi(x_0) \\
& = & -(A - \delta)\phi(x_0) > -A\phi(x_0),
\end{eqnarray*}
a contradiction.  Note that in the second line we have used the second inequality in (G1.3); this is the only place in this work where this inequality is used.

Thus the solution $\phi$ is decreasing for all $x$ and bounded below by zero.  It follows that the limits $\phi(\infty)$ and $\phi'(\infty)$ exist, that $\phi'(\infty) = 0$, and that $\phi(\infty) \in [0,1)$.  Thus $\phi(\infty) = g(\Phi(\infty))$, and it follows from (G1.1) that $\phi(\infty) = 0$.  This completes the proof.
\end{proof}

We have established the existence and uniqueness of a monotone front in the $c > c(\beta)$ case.  To 
show uniqueness among all fronts, as in statement 2 of Theorem $\ref{thm:main}$, we must rule out the possibility that $\phi_+(\infty) = 0$.  This is the purpose of hypothesis (G2).   Since $\phi_+$ satisfies $\eqref{eq:negbc}$ but is not identically one, there must be some $x_0 \in \R$ such that $\phi_+(x_0) > 1$ and $\phi_+(x_0) \geq \phi_+(x_0 + s)$ for all $s \in [-r,0]$.  In this case, though, hypothesis (G2) yields that $\phi_+'(x_0) > 0$.  The conditions just articulated will be preserved as $x$ moves forward; thus in fact $\phi_+(x)$ is strictly increasing for all $x \geq x_0$.  

Thus we have

\begin{proposition}
If $c > c(\beta)$, $\phi_-$ is the unique monotone front connecting $1$ to $0$.  If in addition (G2) holds, $\phi_-$ is the unique front connecting $1$ to $0$. 
\end{proposition}

Let us write 
\[
c_m = \inf \{ \bar{c} > 0 \ : \ \mbox{every initially decreasing front for $c > \bar{c}$ is monotone} \ \}
\]
(this is the $c_m$ of the theorem.)  

We now show that $\mathcal{M}$ is closed and contains $[c(\beta),\infty)$ as in statement 3 of Theorem $\ref{thm:main}$.
\begin{proposition} $\mathcal{M}$ is closed and contains $[c(\beta),\infty)$.  In particular, when $c = c_m$, the solution $\phi_-$ described in Proposition $\ref{lem:W^u}$ is a monotone front.  \label{prop:lim}
\end{proposition}
\begin{proof}
Let $c_k$ be a sequence of points in $\mathcal{M}$ with limit $c_0$ and let $\phi_k$ denote the corresponding monotone fronts.  Let us translate each $\phi_k$ so that $\phi_k(0) = 1/2$.  We know from Proposition \ref{prop:uccs} that $\phi_k \to \phi_0$, where the convergence is uniform on compact subsets and $\phi_0$ is a solution of \eqref{eq:wpe}.  Since each $\phi_k$ is monotone, $\phi_0$ is monotone as well and hence has limits at $x = \pm \infty$; additionally the limits $\phi_0'(\pm \infty)$ exist and are zero.  Thus $\phi_0(\pm \infty) = g(\Phi_0(\pm \infty))$.  The pinning $\phi_k(0) = 1/2$ yields that $\phi_0(-\infty) \in [1/2,1]$ and $\phi_0(\infty) \in [0,1/2]$.  It now follows from (G1.1) that $\phi_0(-\infty) = 1$ and $\phi_0(\infty) = 0$.  Since $\phi_0(-\infty) = 1$ and $\phi_0(x)$ is never greater than one, it follows that $\phi_0 = \phi_-$ and that (up to translation) $\phi_0$ is the unique decreasing front.  
\end{proof}

We now complete the proof of statement 3 in Theorem $\ref{thm:main}$ by showing that $\mathcal{M}$ does not intersect $[b(\nabla g({\bf 0})),c(\nabla g({\bf 0})))$.

\begin{proposition} \label{prop:lowerlimit}
Let $\phi_-$ denote the initially decreasing branch of $W^u(1)$, and assume that $\phi_-$ is a decreasing front.  Then $c \geq c(\nabla g({\bf 0}))$.  In particular, $c_m \geq c(\nabla g({\bf 0}))$.  
\end{proposition}

The ideas in the following proof are similar to those in sections $7$ and $10$ of \cite{mallet-paret:1988}.  We rely on the fact that solutions of the linearized equation \eqref{eq:lwpe} have expansions in eigenfunctions (see, for example, section V.3 of \cite{diekmann:1995}, and \cite{henry:1970}).  

\begin{proof}
Let $y_n(x) := \frac{\phi_-(x+n)}{\phi_-(n)}$.  We shall write 
\begin{eqnarray*}
& & \Phi_-(x) = (\phi_-(x),\phi_-(x-\kappa_1),\ldots,\phi_-(x - \kappa_N)); \\
& & Y_n(x) = (y_n(x),y_n(x-\kappa_1),\ldots,y_n(x - \kappa_N)).
\end{eqnarray*}

Observe that each function $y_n$ satisfies the delay equation
\begin{eqnarray*}
y_n'(x) & = & \frac{1}{c} \left[ - \frac{1}{\phi_-(n)}g(\phi_-(n)Y_n(x)) \right] \\
& = & \frac{1}{c} \left[- \nabla g({\bf 0}) \cdot Y_n(x) + \delta_n(x)\right] 
\end{eqnarray*}
where $\delta_n(x) =  \nabla g({\bf 0}) \cdot Y_n(x) - \frac{1}{\phi_-(n)}g(\phi_-(n)Y_n(x))$.

Since $\phi_-$ is hypothesized to be a decreasing front, we have $y_n(x) \in [0,1]$ for all $x \geq 0$.  Note also that, for $n$ sufficiently large, $y_n'(x)$ is approximated by 
\[
\frac{- \nabla g({\bf 0}) \cdot Y_n(x+n)}{c},
\]
and so the $y_n$ are an equicontinuous family for $x \geq 0$.  Therefore some subsequence of $\{y_n\}$ converges uniformly on compact subsets of $[0,\infty)$ to some $y_0$.  The remainder term $\delta_n$ satisfies $\delta_n = \mathcal{O}(\phi_-(n))$ on the half-line $x \in[0,\infty)$.  Thus $y_0$ 
satisfies the linear constant-coefficient delay equation
\be 
cy_0'(x) = - \nabla g({\bf 0}) \cdot Y_0(x) \qquad x \ge 0, \label{eq:ystar}
\ee
which is simply equation \eqref{eq:lwpe} restricted to a half-line.    

The solution $y_0$ has an expansion in eigenfunctions.  Looking at the real and imaginary parts of the equation $D(\lambda; c, \nabla g({\bf 0})) = 0$ (recall the proof of Proposition \ref{lem:cdef}) shows that, if $x_0 + iy_0$ is a nonreal root of $D$, the only roots of $D$ with real part $x_0$ are $x_0 + iy_0$ and $x - iy_0$.  Moreover, since $D$ is entire its roots are isolated.  Thus there are at most two dominant terms in the expansion of $y_0$; these are either oscillatory (if the corresponding roots are nonreal) or monotone (if the corresponding root is real).   

Since each $y_n$ is decreasing, $y_0$ is non-increasing.  There are now two possibilities.  Either $y_0$ is identically zero, or the root corresponding to the leading term in the expansion of $y_0$ is negative and real.  But $y_n(0) \equiv 1$ for all $n$, so $y_0(0) = 1$ and $y_0$ cannot be identically zero.  This completes the proof.
\end{proof}

We have proven parts 2 and 3 of Theorem \ref{thm:main}.  

\section{$\mathcal{U}$ is open}

In this section we complete the proof of Theorem $\ref{thm:main}$ by proving statement 4.  In particular, we show that if $c \in \mathcal{U}$ and $c > b(\nabla g({\bf 0}))$, then $\mathcal{U}$ contains an open interval in $\R$ about $c$.  We use a continuation argument.  We assume that (G1) holds throughout, and in addition assume that $g$ is $C^1$.  

We write $W^{1,\infty} = W^{1,\infty}(\R)$ for the Banach space of Lipschitz functions on $\R$, equipped with the norm 
\[
\|z\| = \mathrm{esssup}_{\R} |z| + \mathrm{esssup}_{\R} |z'|,
\]
and $L^\infty = L^\infty(\R)$ for the Banach space of essentially bounded functions equipped with the norm
\[
\|z\| = \mathrm{esssup}_{\R} |z|.  
\]
Given $\nu \in \R$, we write $W^{1,\infty}_{\nu} \subset W^{1,\infty}$ for the codimension-one affine subspace of functions $z$ satisfying $z(0) = \nu$.   Given $z$, as we done have earlier we shall write 
\[
Z(x) = \left(z(x), z(x - \kappa_1),\ldots,z(x - \kappa_N) \right).  
\]
We define the following function $F: W^{1,\infty}(\R) \times \R_+ \to L^\infty$:
\be 
F(z,c)(x) := cz'(x) + g(Z(x)). \label{eq:Fdef}
\ee
Note that zeros of $F$ are solutions of $\eqref{eq:wpe}$ with speed $c$ that are defined on all of $\R$.  

Suppose that $\nu \in (0,1)$ and that $\phi \in W^{1,\infty}_{\nu}$ is an initially decreasing front of speed $c_0$.  Below we linearize $F$ about $(\phi,c_0)$ and show that the resulting operator is an isomorphism between $W^{1,\infty}_0$ and $L^\infty$.  $W^{1,\infty}_0$ is the tangent space to $W^{1,\infty}_{\nu}$; therefore by the Implicit Function Theorem $F(\cdot,c)$ has a zero $\phi_c$ in $W^{1,\infty}_{\nu}$ for each $c$ sufficiently close to $c_0$.  This $\phi_c$ is a solution of \eqref{eq:wpe} with speed $c$ that is defined for all $x \in \R$, is translated so that $\phi_c(0) = \nu$, and is close to $\phi$ in the $W^{1,\infty}$ norm.  Proposition $\ref{prop:bdycond}$ establishes that $\phi_c$ satisfies the boundary conditions $\phi_c(-\infty) = 1$ and $\phi_c(\infty) = 0$ --- and so $\phi_c$ is a front.  Thus we see that fronts connecting $1$ to $0$ (no longer necessarily monotone) exist for an open subset of wave speeds --- in particular, for some wave speeds $c < c_m$ (and so non-monotone fronts connecting $1$ to $0$ must exist).

The derivative of $F$ with respect to its functional coordinate is given by
\[ 
D_1 F(z,c)\bar{z} = c\bar{z}'(x) + \nabla g(Z(x)) \cdot \bar{Z}(x). 
\]
Observe that the map $(z,c) \mapsto D_1 F(z,c) \in \mathcal{L}(W^{1,\infty},L^\infty)$ (where the latter space is endowed with the usual operator norm) is continuous.  (This is where we use the fact that $g$ is $C^1$, rather than $C^1$ only near ${\bf 0}$ and ${\bf 1}$.)  Therefore, if we establish that $D_1F(\phi,c_0)$ is an isomorphism from $W^{1,\infty}_0$ to $L^\infty$, the Implicit Function Theorem applies.  

\begin{proposition} \label{prop:iso}
Let $c_0 > b := b(\nabla g({\bf 0}))$ and suppose that there is some initially decreasing $\phi \in W^{1,\infty}_\nu$(not necessarily monotone) that satisfies $F(\phi,c_0) = 0$ and the boundary conditions $\phi(-\infty) = 1$ and $\phi(\infty) = 0$.  Suppose further that $\nu$ has been chosen so that, whenever $x < 0$, $\phi(x) \in (\nu,1)$, $\frac{\partial g}{\partial s_0}(\Phi(x)) < 0$, $\sum_{i=0}^N \frac{\partial g}{\partial s_i}(\Phi(x)) < 0$, and 
\[
\nabla g(\Phi(x)) \cdot (s_0,\ldots,s_N) < 0 \ \mbox{when} \ 0 \leq s_i \leq s_0 \ \mbox{for all} \ 1 \leq i \leq N.  
\].  

Then $L := D_1 F(\phi,c_0)$ is an isomorphism from $W^{1,\infty}_0$ to $L^\infty$.
\end{proposition}

\begin{remark}
In light of (G1.2) and the fact that $g$ is $C^1$, the conditions on $\nabla g(\Phi(x))$ for $x < 0$ can be imposed without loss of generality by choosing $\nu$ appropriately. 
\end{remark}

\begin{proof}
Since $c_0 > b$ the characteristic equation of $\eqref{eq:wpe}$ at 0 has no roots in the closed right half plane.  Since $c_0 > 0$, the characteristic equation at 1 has one root in the right half-plane.  Thus, it follows from Theorem A in \cite{mallet-paret:1999b} that the Fredholm index of $L$ is $1$ when $L$ is regarded as an operator from $W^{1,\infty}$ to $L^\infty$.  Since $W^{1,\infty}_0$ is codimension one in $W^{1,\infty}$, the Fredholm index of $L$ is zero when regarded as an operator from $W^{1,\infty}_0$ to $L^\infty$.  Thus to show that $L$ is an isomorphism it suffices to show that its kernel in $W^{1,\infty}_0$ is $\{0\}$.  Suppose that $w \in W^{1,\infty} \setminus \{0\}$ and that $Lw = 0$.  We will show that $w \not \in W^{1,\infty}_0$.  

The expression $Lw = 0$ rewrites as
\be
cw'(x) = - \nabla g(\Phi(x)) \cdot W(x) \label{eq:lin}, \ \ W(x) = (w(x),w(x-\kappa_1), \ldots, w(x-\kappa_N)).  
\ee
We claim that $w$ satisfies the boundary condition $w(-\infty) = 0$.   In fact this claim follows from the theory of exponential dichotomies, but for completeness we provide an elementary proof here.  There are two cases to consider.  Either $w$ is initially monotone, or it isn't.  If $w$ is initially monotone, then the limits $w(-\infty)$ and $w'(-\infty)$ exist and $w'(-\infty) = 0$.  Thus $0 = \nabla g({\bf 1})\cdot(w(-\infty),\cdots w(-\infty))$.  Since the sum of the entries of $\nabla g({\bf 1})$ is negative, we have $w(-\infty) = 0$ as desired.  We now consider the case where $w$ is not initially monotone.  There is a sequence $x_k \to -\infty$ with $w'(x_k) = 0$.  Choose a subsequence so that $w(x_k)$ and $w(x_k-\kappa_i)$ simultaneously converge, for all $i$, to limits $w_* = \limsup_{x \to -\infty} w(x)$ and $w_i$.  Then, letting $k \to \infty$ in the expression $0 = \nabla g(\Phi(x_k)) \cdot W(x_k)$, we obtain 
\[
0 = \nabla g({\bf 1}) \cdot (w_*,w_1, \ldots, w_N).
\]
Imagine that $w_* > 0$.  Then 
\[
\frac{\del g({\bf 1})}{\del s_i} w_i \leq \frac{\del g({\bf 1})}{\del s_i} w_* 
\]
for all $i \in \{1,\ldots,N\}$, whence 
\[
0 \le w_* \sum_{i=0}^N \frac{\partial g}{\partial u_i}({\bf 1}),
\]
a contradiction.  The case $w_* < 0$ is ruled out similarly.  Thus the limit exists and the claim is proven.

Since $\lim_{x \to -\infty} w(x) = 0$, the solution $w$ lies on the unstable manifold at $0$ for the linear equation $\eqref{eq:lin}$.  Thus, from an argument similar to the proof of Lemma $\ref{lem:W^u}$ except with $\langle \nabla g(\Phi(x)), \cdot \rangle$ replacing $g(\cdot)$, we know that $w$ is initially strictly of one sign.  Without loss of generality we take $w(x) > 0$ for all sufficiently negative $x$.  Choose $x_0 < 0$ such that $w(x) > 0$ for all $x \leq x_0$ and also such that 
\[
w(x_0) \geq w(x_0 - \tau), \ \tau \in [0,r]  
\]
(such an $x_0$ exists because $w(-\infty) = 0$).  Our hypotheses on $\nabla g(\Phi(x))$ for $x < 0$ now yield that $w'(x_0) > 0$. 

Now let 
\[
x_* := \sup\{ x \in \R \: | \; w'(y) > 0 \mbox{ for all } y \in [x_0,x) \}.
\]
If $x_*$ is infinite, then $w \not \in W^{1,\infty}_0$ and we are done.  Suppose that $x_*$ is finite.
Then $w'(x_*) = 0$ so $\nabla g(\Phi(x_*)) \cdot W(x_*) = 0$.  Since $w(x)$ is strictly increasing on $[x_0,x_*]$, by our hypothesis $w'(x) > 0$ on $[x_0,x_*) \cap (-\infty,0]$.  It follows that $x_* > 0$.  Thus in particular $w(x) > 0$ for $x \le 0$, and $w \not \in W^{1,\infty}_0$.  This completes the proof.
\end{proof}

We now complete the proof that $\mathcal{U}$ is open.  Suppose that $c_0 \in \mathcal{U}$, with a corresponding front $\phi$.  Apply the Implicit Function Theorem with $c$ as a parameter: for each $\epsilon > 0$ there is a $\delta = \delta(\epsilon) > 0$ such that whenever $|c - c_0| < \delta$ then there is a $\phi_c$ which solves $\eqref{eq:wpe}$ and which satisfies $\| \phi_c - \phi_0 \| < \epsilon$.   Moreover, since $\phi_c$ is continuous in $c$ we have $\|\phi_c - \phi_0 \| \to 0$ as $\delta \to 0$.  

It remains to check that $\phi_c$ connects $1$ to $0$ if $\delta$ is small enough.  Given any $\epsilon$, by choosing $c$ close enough to $c_0$ we can arrange for $\|\phi_c - \phi\| < \epsilon/2$ in the $W^{1,\infty}$ norm, and so also in the sup norm.  By shifting $\phi_c$ and $\phi$ (identically) so that $|\phi(x) - 1| < \epsilon/2$ for $x < 0$, we guarantee that $|\phi_c (x) - 1| < \epsilon$ for all $x < 0$.  Similarly, we can shift $\phi_c$ so that $|\phi(x)| < \epsilon$ for all $x > 0$.  That this is enough to guarantee $\phi_c (-\infty) = 1$ and $\phi_c(\infty) = 0$ follows from the following proposition.

\begin{proposition} \label{prop:bdycond}
For each $c_0 > b(\nabla g({\bf 0}))$, there is an $\eps > 0$ such that the following hold.  
\begin{itemize}
\item[1)] Suppose that $|c - c_0| < \eps$ and that $\psi_0 \in C$ has a global backward continuation for $\eqref{eq:wpe}$, with wave speed $c$, such that $\| \psi_{\tau} - 1 \| < \eps$ for $\tau < 0$.  Then the limit $\lim_{t \to -\infty} \psi(t)$ exists and is equal to $1$.
\item[2)] Suppose that $|c - c_0| < \eps$ and that $\psi_0 \in C$ has a global forward continuation for $\eqref{eq:wpe}$, with wave speed $c$, such that $\| \psi_{\tau}  \| < \eps$ for $\tau > 0$.  Then the limit $\lim_{t \to \infty} \psi(t)$ exists and is equal to $0$.
\end{itemize}
\end{proposition}

\begin{proof}
Consider the dynamical system 
\[ 
\left\{ \ba{l} -c\phi'(x) = g(\Phi(x)) \\ \dot{c} = 0 \ea \right. 
\]
in the phase space $C \times \R$.  Choose $\delta > 0$ such that the set $W_{loc}^c(1,c_0) = \{ (1,c) \; | \; |c-c_0| < \delta \}$ is a local center manifold of the equilibrium $(1,c_0)$.  

By Theorem X.2.1 in \cite{hale:1993}, we can choose $\eps < \delta$ so small that every invariant set in an $\eps$-neighborhood of $(1,c_0)$ is contained in the local center manifold of $(1,c_0)$.  $(\psi_0,c)$ has a backward orbit that is contained in an $\eps$-neighborhood of $(1,c)$.  Since $\psi(t)$ is $C^1$ and bounded with bounded derivative for all negative $t$, this backward orbit is pre-compact and so has a well-defined $\alpha$-limit set.  This $\alpha$-limit set is invariant and contained in an $\eps$-neighborhood of $(1,c_0)$, and hence is a subset of $W_{loc}^c(1,c_0)$.  Since $c$ is constant, it follows that $\alpha(\psi_0,c) = (1,c)$.  In particular it follows that $\lim_{t \to -\infty} \psi(t) = 1$, as desired.  This completes the proof of the first statement.

The proof of the second statement is similar.  Choose $\delta > 0$ such that the set $W_{loc}^c(0,c_0) = \{ (0,c) \; | \; |c-c_0| < \delta \}$ is a local center manifold of the equilibrium $(0,c_0)$.  Since the $\omega$-limit set of $(\psi_0,c)$ is invariant and uniformly close to $0$, it is a subset of $W^c_{loc}(0,c_0)$.  Since $c$ is constant, it follows that $\lim_{t \to \infty} \psi(t) = 0$, as desired.
\end{proof}

To complete the proof of part 4 of Theorem \ref{thm:main}, it remains to show that $\phi_c$ is in fact initially decreasing --- that is, we have to rule out that $\phi_c$ is not the ``positive" branch of $W^u(1)$.  We only sketch the argument; it is similar to that at the close of the proof of Proposition $\ref{prop:iso}$.  If $\phi_c$ is not initially decreasing, it is initially greater than $1$.  Given any $\eta > 0$, we can choose $x_0$ such that $0 < \phi(x)-1 < \eta$ for all $x \leq x_0$ and such that 
\[
\phi_c(x_0) \geq \phi_c(x_0 - \tau), \ \tau \in [0,r]. 
\]
Since $g$ is $C^1$ on a neighborhood of ${\bf 1}$, our hypotheses on the gradient of $g$ at ${\bf 1}$ imply that, if $\eta$ is small enough, then $\phi_c'(x_0) > 0$.  $\phi'(x)$ will similarly be strictly positive, for $x > x_0$, at least until $\phi(x) = \eta$.  In particular, $\phi_c$ is not close to $\phi_{c_0}$, a contradiction.  

This completes the proof that $\mathcal{U}$ is open and hence completes the proof of Theorem $\ref{thm:main}$.

\section{Behavior at $c_* \in \del \mathcal{U}$}

We now turn to the proof of our second main theorem.  

\begin{proposition}
Assume that $g$ is $C^1$ and that (G3) holds.  For each  $c_* \in \partial \mathcal{U}$ with $c_* > b(\nabla g ({\bf 0 }))$ there are two globally bounded non-constant solutions to $\eqref{eq:wpe}$, $\psi^0$ and $\psi^1$, neither of which are connections from $1$ to $0$, such that $\psi^0(\infty) = 0$ and $\psi^1(-\infty) = 1$.
\end{proposition}

\begin{proof}
Let $c_* \in \del \mathcal{U}$, choose $\{c_k\} \subset \mathcal{U}$ with $c_k \to c_*$ and let $\phi_k$ denote the corresponding initially decreasing fronts with wave speed $c_k$.  The $\phi_k$ are uniformly bounded by some $M$ by Lemma $\ref{lem:G3}$.  By Proposition \ref{prop:uccs}, a uniform-on-compact-sets limit $\phi_*$ is a solution of \eqref{eq:wpe} that is defined for all $x \in \R$.  Since $\mathcal{U}$ is open and $c_* \in \del \mathcal{U}$ by hypothesis, $\phi_*$ does not connect $1$ to $0$ --- that is, $\phi_*$ does not satisfy both boundary conditions $\phi_*(-\infty) = 1$ and $\phi_*(\infty) = 0$.  We claim, though, that by translating the $\phi_k$ properly we can arrange for $\phi_*$ to satisfy either of these conditions.  

Since all of the $\phi_k$ are initially monotone and decreasing until they cross $0$, we can shift all of the $\phi_k$ to satisfy $\phi_k(0) = 1/2$, with $\phi_k$ decreasing for $x \leq 0$.  Then $\phi_*$ satisfies these same conditions, and $\phi_*(-\infty) = 1$.  In this case, $\phi_*$ is the $\psi^1$ of the theorem.  

Recall that we are assuming that $c_* > b(\nabla g({\bf 0}))$.  By Proposition \ref{prop:unif_basin}, then, there is some $\epsilon$ such that for all large $k$, the basin of attraction of $0$ for $c_k$ (and $c_*$) contains the ball in $C$ of radius $\epsilon$.  Each $\phi_k$ approaches $0$ and so there is, for each $k$, a final time when $|\phi_k(x)| = \frac{\epsilon}{2}$.  Translate each $\phi_k$ to make this final time at time $0$.  Then $\phi_*(0) = \pm \epsilon/2$, and $|\phi_*(x)| \leq \frac{\epsilon}{2}$ for all $x \in [0,1]$.  By Proposition $\ref{prop:unif_basin}$, $\phi_*$ must approach $0$ as $x \to \infty$.  $\phi_*$ is not, however, identically zero.  In this case, $\phi_*$ is the $\psi^0$ of the theorem.   
\end{proof}

\section{Examples}\label{EXAMPLES}

In section \ref{subsec:ex} we mentioned \cite{peletier:2004}, in which Peletier and Rodriguez study the equation 
\begin{equation}\label{eq:pareqcov}
u_n' = -u_n + h(u_{n-1}), \ \ h(u) = \left\{\begin{array}{cc} 2u, & u \leq 1/2 \\ 1, & u \geq 1/2 \end{array} \right.
\end{equation}
(we have changed variables to put \eqref{eq:pareqcov} into the framework of equation \eqref{MODEQ}).  In our notation, for this equation $c(\nabla g({\bf 0})) = c(\beta) = c(-1,2) \approx 4.31107$.  Using a mixture of numerical and analytical techniques, in \cite{peletier:2004} the authors explore the behavior of fronts as $c$ decreases.  In particular, they find values $0 < c_{unb}  < c_{bif} < c(-1,2)$ such that initially decreasing solutions of \eqref{eq:wpe} with $\phi(-\infty) = 1$ 
\begin{itemize}
\item approach $0$ monotonically for $c \geq c(-1,2)$; 
\item approach $0$ nonmonotonically for $c \in (c_{bif},c(-1,2))$; 
\item are bounded but do not approach $0$ for $c \in [c_{unb},c_{bif}]$;
\item are unbounded for $c < c_{unb}$.
\end{itemize}
The value $c_{bif}$ is where the leading roots of the characteristic equation cross the imaginary axis.  For this example, therefore, in our notation we have $c_m = c(\nabla g({\bf 0}))$ and $c_f = b(\nabla g({\bf 0}))$ --- otherwise put, the bifurcation values $c_m$ and $c_f$ are exactly what one would expect given only the linearization of \eqref{eq:wpe} at $0$.  

We mention again the work of Hsu, Lin, and Shen in \cite{hsu:1999}, which includes a thorough discussion of the behavior of traveling wave solutions of various speeds for a particular family of lattice differential equations with unidirectional coupling.  In the equations considered in \cite{hsu:1999} there are three equilibria $\mu^- < \mu^0 <\mu^+$.  For wave speeds large enough there is a monotone traveling wave solution connecting $\mu^0$ to $\mu^+$.  Under appropriate conditions, as the wave speed decreases the sequence of wave behaviors is reminiscent of that for \eqref{eq:pareqcov}: we have monotone waves connecting $\mu^0$ to $\mu^+$, then bounded nonmonotone waves oscillating about $\mu^0$ as $x \to -\infty$, then (for a particular wave speed) a monotone wave connecting $\mu^-$ to $\mu^+$, and then waves satisfying $\phi(\infty) = \mu^+$ that are unbounded below.  Somewhat more generally, there is a range of wave speeds such that there is no connection from $\mu^0$ to $\mu^+$, but such that waves are bounded.  (See Theorem A in \cite{hsu:1999}).  

Similarly, to return to the discussion of \eqref{eq:pareqcov}, the range $c \in [c_{unb},c_{bif}]$ corresponds to bounded traveling waves that do not connect $1$ to $0$.  What is apparently happening here (expressed in terms of our notation) is that, as $c$ drops below $c_f$, the initially decreasing branch $\phi_-$ of the unstable manifold at $1$ (what is called, in the $c = c_f$ case, $\psi^1$ in Theorem \ref{thm:main2}) is connecting $1$ to a periodic orbit about $0$.  (Indeed, for \eqref{eq:pareqcov}, Peletier and Rodriguez find a subinterval of $[c_{unb},c_{bif}]$ for which waves are eventually periodic, and there is an analogous interval of wave speeds for some of the class studied in \cite{hsu:1999}.  Waves that are literally eventually periodic, though, arise as artifacts of the piecewise linear feedback functions used in both \cite{peletier:2004} and \cite{hsu:1999}.)  It seems, though, that there should be examples of \eqref{MODEQ} for which initially decreasing solutions of (\ref{eq:wpe},\ref{eq:negbc}) are unbounded for all $c < c_f$, and examples for which $\psi^1$ connects $1$ to a nonzero equilibrium solution.   

Our goal in this section, broadly speaking, is to suggest the breadth of possible behaviors of solutions of (\ref{eq:wpe},\ref{eq:negbc}) as $c$ varies.  We will confine our attention to the following specialized single-delay problem: 
\begin{equation}\label{eq:odv}
c \phi'(x) = \phi(x) - f(\phi(x-1)). 
\end{equation}
In this particular context, the assumptions (G1) may be expressed as follows (recall section \ref{subsec:ex}): $f$ is a $C^1$ function such that
\begin{itemize}
\item $f(0) = 0$, $f(1) = 1$, and $f(s) \ne s$ for $s \in (0,1)$;
\item $f'(0) > 1 > f'(1) \geq 0$;
\item There is some $\beta_1 > 1$ such that $ s < f(s) < \beta_1 s$ for $s \in (0,1)$.  
\end{itemize}

Note that, in this context, since $g(s_0,s_1) = -s_0 + f(s_1)$ we have 
\[
c(\nabla g ({\bf 0})) = c(-1,f'(0)) \ \mbox{and} \ b(\nabla g ({\bf 0})) = b(-1,f'(0)).
\]
Note also that, for \eqref{eq:odv}, it is sufficient for (G3) to hold that $f$ be bounded.  

We present the following three examples.  The first two show that the linearization at $0$ does not, in general, provide exact information about the locations of $c_m$ and $c_f$.  

\begin{itemize} 
\item[A:] Given $f'(0)$ and $z > 0$, $f$ can be chosen so that $c_m > c(-1,f'(0)) + z$.
\item[B:] $f$ can be chosen so that $c_f$ lies strictly between $b(-1,f'(0))$ and $c_m$. 
\item[C:] Non-initially decreasing fronts can indeed exist.  
\end{itemize}

\subsection{Example A}

Let $\alpha > 1$ and $z > 0$ be given and let $\gamma$ be so large so that $c(-1,\gamma) > c(-1,\alpha) + z$.  We will define a two parameter family of functions $f_{\tau,\eta}$ which satisfy (G1) as follows.  Choose and fix functions $f_{0,0}$ and $f_{1,1}$ which satisfy (G1) with $f_{0,0}'(0) = \alpha$ and $f_{1,1}'(0) = \gamma$.  Moreover, let $f_{0,0}$ and $f_{1,1}$ be concave down so that $\alpha$ and $\gamma$ also play the role of $\beta_1$ for $f_{0,0}$ and $f_{1,1}$, respectively.  It follows from Theorem $\ref{thm:main}$ that for $c \ge c(-1,\alpha)$ the lattice equation $\eqref{eq:odv}$ with $f$ given by $f_{0,0}$ has a monotone front connecting $1$ to $0$.  On the other hand, for $c < c(-1,\gamma)$, the equation $\eqref{eq:odv}$ with $f$ given by $f_{1,1}$ has no monotone front connecting $1$ to $0$ (Proposition \ref{prop:lowerlimit}).  

For $(\tau,\eta) \in (0,1] \times [0,1]$, let $\rho_{\tau,\eta}$ be a smooth cutoff function which takes values in $[0,1]$, is increasing, and satisfies
\[ 
\rho_{\tau,\eta}(s) = \left\{ \ba{ll} 0 & s \le \tau/2; \\ \\ \eta & s \ge \tau. \ea \right. 
\]
Define a two-parameter family of feedback functions $f_{\tau,\eta}$ by
\[ 
f_{\tau,\eta}(s) := (1-\rho_{\tau,\eta}(s))f_{0,0}(s) + \rho_{\tau,\eta}(s)f_{1,1}(s)
\] 
--- that is, $f_{\tau,\eta}$ is equal to $f_{0,0}$ for $s \le \tau/2$ and equal to the convex combination ``$\eta$ of the way from $f_{0,0}$ to $f_{1,1}$" for $s \ge \tau$.  Such an $f_{\tau,\eta}$ is illustrated in Figure 1.  

\begin{center}
\includegraphics[width = 6in, height = 4in]{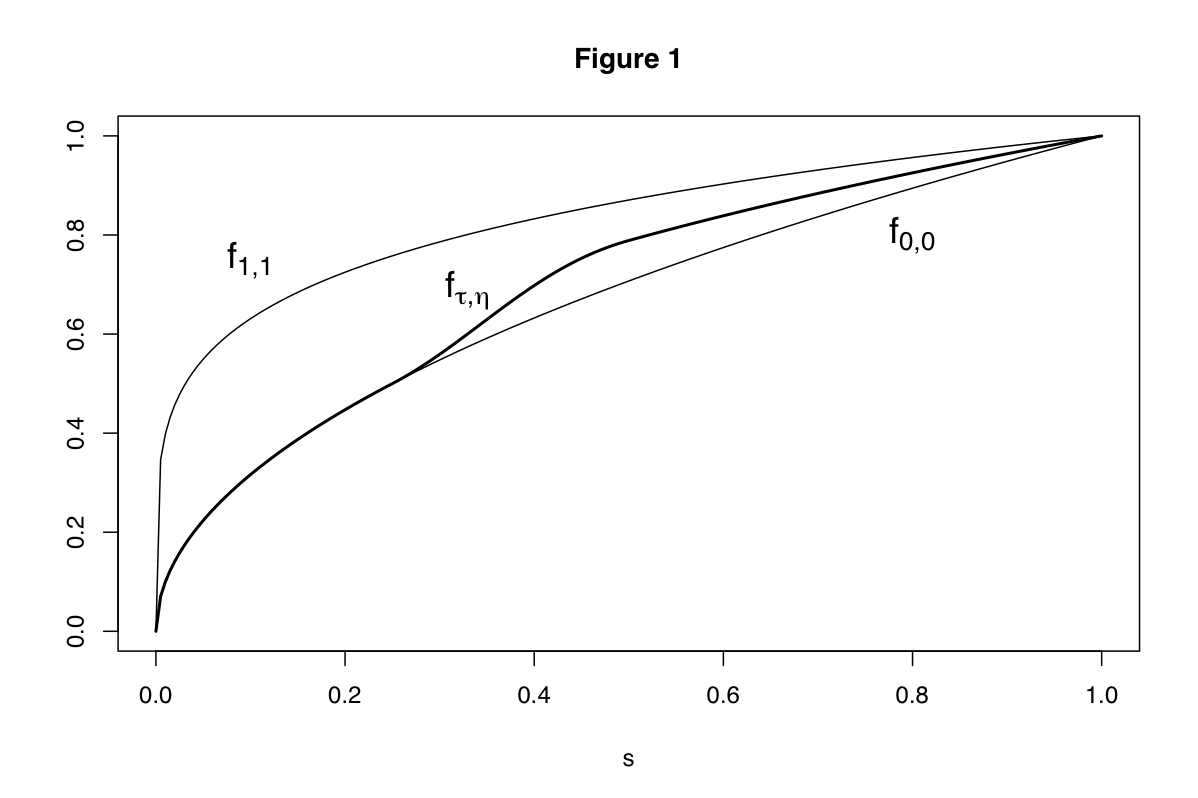}
\end{center}

Note that $f_{\tau,\eta}'(0) \equiv \alpha$ for $(\tau,\eta) \in (0,1] \times [0,1]$.  When $\eta = 0$ we have $f_{\tau,0} \equiv f_{0,0}$ and thus $\eqref{eq:odv}$ with $f = f_{\tau,0}$ admits a monotone front for $c \ge c(-1,\alpha)$.  

We claim the following: For each $c \in [c(-1,\alpha),c(-1,\gamma))$ --- in particular, for some $c > c(-1,\alpha) + z$ --- there is some $\tau_*$ and some $\eta \in [0,1]$ such that $\eqref{eq:odv}$ with wave speed $c$ and $f = f_{\tau_*,\eta}$ admits no monotone front connecting $1$ to $0$.  For otherwise, an argument similar to that in the proof of Propositions \ref{prop:uccs} and $\ref{prop:lim}$ would imply that $\eqref{eq:odv}$ with wave speed $c$ and $f = f_{1,1}$ admits a monotone front connecting $1$ to $0$; this contradicts Proposition $\ref{prop:lowerlimit}$.  (The proof of Proposition \ref{prop:uccs} has to be modified to apply to a convergent sequence of feedback functions rather than a convergent sequence of wave speeds.) 

The point is that $f'(0)$ alone provides only a lower bound on $c_m$; in general, all we know about $c_m$ is that it lies somewhere in the interval $[c(-1,f'(0)), c(-1,\beta_1)]$.  If $f$ is concave down, of course, since $\beta_1$ can be taken equal to $f'(0)$ we necessarily have $c_m = c(-1,f'(0))$.  

\subsection{Example B}

In this example we show that $c_f$ can be strictly greater than $b(\nabla g({\bf 0})) = b(-1,f'(0))$.  Similarly to the last example, the point is that the linearization of \eqref{eq:odv} at $0$ does not determine the range of wave speeds for which fronts exist.  (We remark that, while it seems plausible both that $\mathcal{U}$ is connected and that $\mathcal{U}$ has lower bound $b(-1,f'(0))$, we have not proven either of these claims.)   Since our examples will use smooth bounded feedback functions $f$, hypothesis (G3) holds and $c_f > b(-1,f'(0))$ implies that the solutions $\psi^1$ and $\psi^0$ described in Theorem \ref{thm:main2} exist.  

We first consider equation \eqref{eq:odv} with feedback function $f = f_\epsilon$, where $\epsilon \in (0,1/2)$ and $f_\epsilon$ is a smooth, odd, increasing function satisfying  
\[
f_\epsilon(s) = \left\{ \begin{array}{cc} 2s, & s \in [0,\epsilon]; \\
1, & s \geq 2\epsilon. \end{array} \right.
\]
For definiteness, let us take $c = 5$.  Observe that that $5 > c(-1,f'(0)) = c(-1,2) > b(-1,2)$.  We show that, for $\epsilon$ small enough, $f = f_\epsilon$, and $c = 5$, the initially decreasing branch $\phi_-$ of the unstable manifold of $\eqref{eq:odv}$ at $1$ coincides, after finite time, with a nontrivial periodic solution oscillating about zero.  

Define
\[ \Sigma := \{ \psi \in C \: | \: \psi(0) = 2\eps \mbox{ and } \psi(x) \ge 2\eps \mbox{ for } x \in [-1,0]\}, \]
where the notation is chosen to be suggestive of a Poincare section.  Define $\Sigma' \subset \Sigma$ to be the subset of $\psi \in \Sigma$ such that $\psi_t \in \Sigma$ for some $t \ge 1$.  Define the first return time
$\tau : \Sigma' \to [1,\infty)$ by $\tau(\psi) = \inf_{t \ge 1} \{t \: | \; \psi_t \in \Sigma\}$.  Define the first return map 
$P : \Sigma' \to \Sigma'$ by $P(\psi) = \psi_{\tau(\psi)}$.  Fixed points of $P$ correspond to periodic orbits of $\eqref{eq:odv}$. 
We claim that $P$ has a fixed point $\phi_*$; moreover, the initially decreasing branch of the unstable manifold at one, which we now denote by $\phi$, coincides, after a finite time, with $\phi_*$.  The two observations, (i) $P$ takes at most one value and (ii) $\phi \in \Sigma'$, imply the existence of a fixed point $\phi_* = P(\phi)$ and the fact that $\phi_t$ is eventually periodic.  To see that $P$ cannot take more than one value, suppose that $\psi^1$ and $\psi^2$ are both in the range of $P$.  Let $\psi^1(t)$ and $\psi^2(t)$ denote their forward continuations under $\eqref{eq:odv}$ with $f = f_\eps$.  Since $f_\eps$ is constant for $x \ge 2\eps$ and $\psi^1,\psi^2 \in \Sigma$, it follows that $\psi^1$ and $\psi^2$ agree on $[0,1]$.  Since $\psi^1(t)$ and $\psi^2(t)$ agree on an interval of length one, it follows that they agree for all $t \ge 0$.

We now establish that $\phi_- \in \Sigma'$.  Observe that (a particular translate of) $\phi$ is given by 
\[
\phi(x) = (2\epsilon - 1)e^{x/5} + 1, \ \ x \leq 0.  
\]
Note that $\phi_0 \in \Sigma$.  To show that $\phi_0 \in \Sigma'$, we claim that there is some $x_* > 0$ such that $\phi_-(x_*) = -2\epsilon$ and $\phi_-(x) \leq -2\epsilon$ for all $x \in [x_*-1,x_*]$; since the feedback in $\eqref{eq:odv}$ is odd when $f = f_\epsilon$, by symmetry we have that $\phi_-(2x_*) = 2\epsilon$ and $\phi_-(x) \geq 2\epsilon$ for all $x \in [2x_* - 1,2x_*]$.  Thus the translation of $\phi_{2x_*} \in \Sigma$ and hence $\phi_0 \in \Sigma'$.

We now prove the claim.  For $x \in [0,1]$ $\phi_-$ satisfies the ODE 
\[
\phi_-'(x) = \frac{\phi_-(x) - 1}{5}
\]
and so is still given by $\phi_-(x) = (2\epsilon - 1)e^{x/5} + 1$.  Therefore we have, for $\epsilon$ small enough, that $\phi_-(x) = -2\epsilon$ for some $x = x_0 \in (0,1)$ --- in particular, 
\[
\phi_-(x) = -2\epsilon \ \mbox{for} \ x_0 := 5\ln\left( \frac{1 + 2\epsilon}{1 - 2\epsilon}\right) < 1.
\]
We also have $\phi_-(1) = (2\epsilon - 1)e^{1/5} + 1 < -2\epsilon$.  For $\epsilon$ small, $x_0$ is small and $\phi_-(1)$ is approximately $-0.22$.  We wish to show that $-1 <\phi_-(x) \leq -2\epsilon$ for all $x \in [x_0,x_0+1]$, for then $\phi_-'(x_0+1) = (\phi_- (x_0+1) + 1)/5 > 0$, and $\phi_-$ will continue to increase until the time $x_*$ in our claim.  For $\epsilon$ small, $\phi_-(x)$ is certainly in the desired range for $x \in [x_0,1]$.  Since $f_\epsilon$ is increasing, for $x \in [1,1+x_0]$ we have that $f_\epsilon(\phi_-(x-1))$ is between $-1$ and $1$.  Thus, crudely, for $x \in [1,1+x_0]$ we have
\[
\frac{\phi_-(x)-1}{5} \leq \phi_-'(x) \leq \frac{\phi_-(x) + 1}{5}.
\]
By taking $\epsilon$ small enough and comparing $\phi_-(x)$ to the sub- and supersolutions suggested by the above inequality on the small interval $[1,1+x_0]$, we obtain the desired result.  Thus, even though $5 > b(-1,f_\epsilon'(0))$, $5$ is clearly less than $c_f$ when $f = f_\epsilon$ and $\epsilon$ is small enough.  

We remark briefly that we have used a very specialized $f_\eps$ that guarantees that $P$ takes only one value in order to simplify our analysis.  However, we believe the phenomenon of fronts limiting to connections between equilibria and periodics to be more robust.  An argument similar to that used above can, in many cases, show that $\phi \in \Sigma'$ and hence that the initially decreasing branch of the unstable manifold exhibits recurrent dynamics.  The presence of a periodic solution as a potential $\omega$-limit set for $\phi$ can, in many cases, be shown to exist via Schauder's Theorem together with the fact that $P$ is compact.

We now turn an example of a different kind of behavior as $c$ reaches and drops below $c_f$ (we have not rigorously proven what we are about to describe).   Consider equation \eqref{eq:odv} with $f = f_\alpha$ smooth, increasing, and satisfying the following, where $\alpha > 0$ and $\epsilon \in (0,\alpha/2)$:
\[
f_\alpha(s) = \left\{ \begin{array}{cc} 1, & s \geq 1/2; \\ 2s, & s \in [-\alpha/2+\epsilon,1/2]; \\ -\alpha, & s \leq -\alpha/2. \end{array} \right.
\]

We conjecture that $\alpha$ and $f_\alpha$ can be chosen so that the following hold:

\begin{itemize}
\item $c_f$ is as close to $c_m = c(-1,2)$ as desired;
\item At $c = c_f$, $\phi(x) = -\alpha$ for all sufficiently large $x$;
\item For $c < c_f$, $\phi(x) \to -\infty$ as $x \to \infty$.  
\end{itemize}
In this example, therefore, $\psi^1$ connects two equilibria (and, borrowing notation from \cite{peletier:2004}, $c_{bif} = c_{unb})$.  

There is clearly an equilibrium solution at $-\alpha$.  Any solution of \eqref{eq:odv} with $\phi(x_0) < -\alpha$ and $\phi(x) \geq \phi(x_0)$ for all $x < x_0$ will satisfy $\phi(\infty)= -\infty$.  Also, as $c \to c_f$ from above, the minimum value of the front $\phi_c$ connecting $1$ to $0$ will decrease continuously in $c$.   If $c_f$ is small enough to allow this minimum value to approach $-\alpha$, then at $c = c_f$ this minimum value is equal to $\alpha$ and $\phi_{c_f}(x) = -\alpha$ for all sufficiently large $x$, and $c = c_f$.   

We have not proven rigorously that $c_f$ is actually low enough for this to work --- that is, we have not ruled out a narrow parameter range for which $\phi_c(x) > -\alpha$ for all $x$ but $\phi_c$ approaches a periodic solution rather than the equilibrium at $0$.  Numerical evidence with $\alpha = 1/4$ and $\epsilon = 0$, though, suggests that $c_f \approx 2.1$ and that for all $c < c_f$ the initially decreasing branch of the unstable manifold at $1$ is unbounded.  Note that $2.1$ is well above $b(-1,2) \approx 1.65$.  We could, in fact, make $c_f$ as close as we like to $c_m = c(-1,2)$ by taking $\alpha$ close to $0$.  

Figure 2 shows the initially decreasing branch of the unstable manifold at $1$ for various values of $c$ near $2.1$.  Note that, on compact subsets, the solutions that ultimately approach $0$ appear to be converging to a solution that approaches $-\alpha$ --- that is, the solution $\psi^1$.  
 
\begin{center}
\includegraphics[width = 6in, height = 4in]{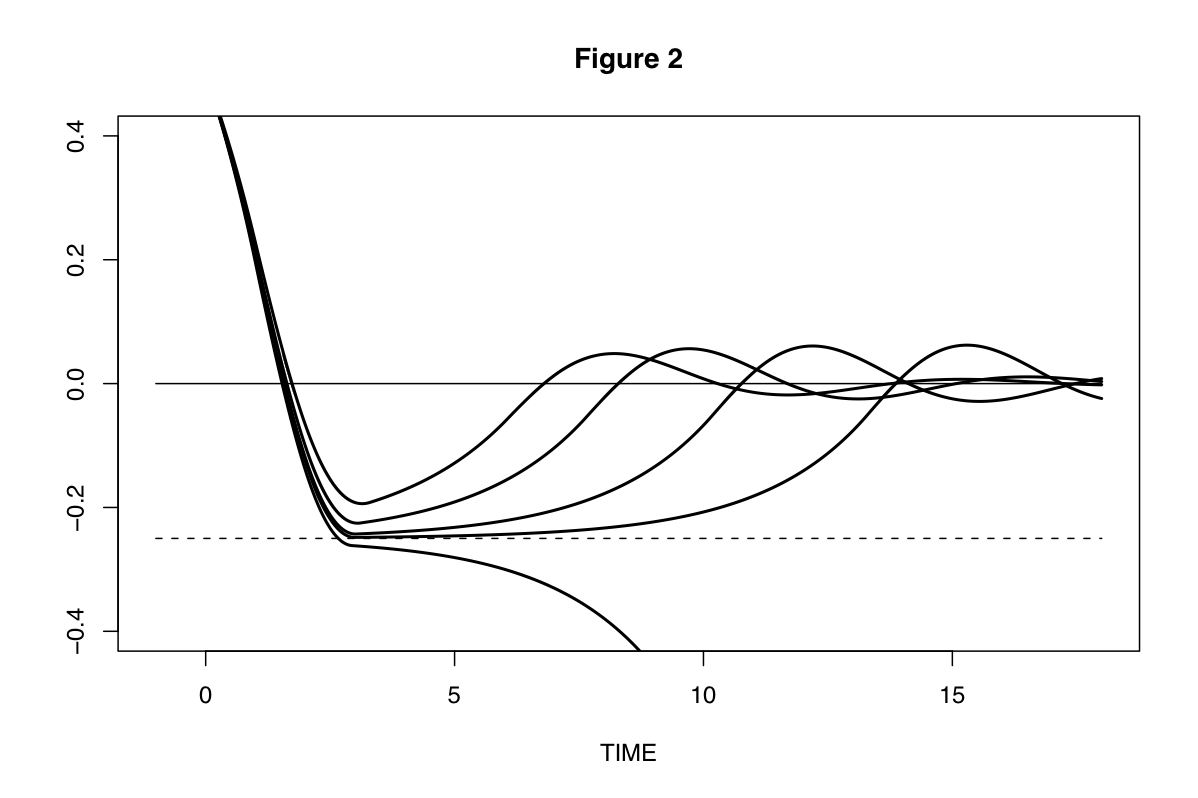}
\end{center}

\subsection{Example C}

While our focus in this paper is on initially decreasing fronts, we emphasize that non-initially decreasing fronts can exist: in the notation of Lemma \ref{lem:W^u}, $\phi_+(\infty) = 0$ is possible.  We present a concrete example, somewhat informally.  

We consider \eqref{eq:odv} where $f$ is an increasing function satisfying 
\[
f(s) = \left\{ \begin{array}{cc} 2s, & s \leq 1/2 - \epsilon; \\ 1, & s \in [1/2,3/2]; \\ 5, & s \geq 3/2 + \epsilon \end{array} \right.
\]
where $\epsilon > 0$ and such that we can take $\beta_1$ slightly larger than $2$.  We take $c = 5 > c(-1,\beta_1))$; in this case there is a decreasing front $\phi_-$.   Since $f(s) = 1$ for $s$ near $1$, for large negative $x$ any solution of \eqref{eq:odv} with $\phi(-\infty) = 1$ satisfies the ordinary differential equation 
\[
\phi'(x) = \frac{\phi(x) - 1}{5}
\]
and is therefore of the form 
\[
\phi(x) = C e^{x/5} + 1.
\]
Taking $C < 0$ yields the initially decreasing branch $\phi_-$ of the unstable manifold, and taking $C > 0$ yields the initially increasing branch $\phi_+$ of the unstable manifold.  In particular let us take the function $\phi_+$ defined by 
\[
\phi_+(x) = \frac{1}{2}e^{x/5} + 1
\]
for all $x \leq 0$.  Observe that $\phi_+(0) = 3/2$.   For $x > 0$ let us continue $\phi_+$ as a solution of the delay equation \eqref{eq:odv}.  We wish to show that $\phi_+(\infty) = 0$.  We will do this by proving the following claim: there is some $x_* > 0$ such that $\phi_+(x_*) < 1$ and $\phi_+(x) \in [1/2 ,3/2]$ for all $x \in [x_*-1,x_*]$.  It follows that, for all $x \geq x_*$, $\phi_+$ coincides with a translate of $\phi_-$ and so is a front.  

$\phi(x)$ will continue to be of the form $(1/2)e^{x/5} + 1$ for $x \in [0, 1]$.   Observe that $\phi(1) = (1/2)e^{1/5} + 1 \approx 1.61$.  Let us take $\epsilon < \phi(1) - 3/2$ so that $x_1 = \inf\{x > 0 \ : \ \phi(x) = 3/2 + \epsilon \ \}$ is less than $1$.  Observe that, for all $x \in [1,x_1 + 1]$, the quantity $\phi(x-1)$ lies between $3/2$ and $3/2 + \epsilon$ and so $f(\phi(x-1))$ lies between $1$ and $5$.  If $\epsilon$ is small, we have that $x_1$ is close to $0$ and that $\phi(1+x_1)$ is close to $\phi(1)$.  In particular, we may assume that $\phi(x) > 3/2 + \epsilon$ for all $x \in (x_1,x_1 + 1]$.  

Let us write
\[
x_2 = \inf \{ \ x > x_1+1 : \ \phi(x) = 3/2 + \epsilon \ \}.
\]
On the interval $[x_1+ 1,x_2+1]$, $\phi(x)$ will satisfy the ODE
\[
\phi'(x) = \frac{\phi(x)}{5} - 1.
\]
In particular, $\phi(x_2 + 1)$ will be equal to 
\[
(\phi(x_2) - 5)e^{1/5} + 5 \approx -\frac{7}{2}e^{1/5} + 5 \approx 0.725.
\]
For $x \in [x_2,x_2 + 1]$, we have (crudely) $\phi'(x) \leq -3$; and for $x \in [x_2 +1,x_2 + 2]$, we have $\phi_+'(x) \geq (\phi_+(x) - 5)/5$.  Therefore, in particular, for $\epsilon$ small enough, $\phi_+(x)$ reaches $3/2$ shortly after $x_2$ and does not reach $1/2$ too soon after $x_2 +1$.  Thus there is some $x_* > x_2 +1$ such that 
$\phi_+(x_*) < 1$ and $\phi_+(x) \in [1/2,3/2]$ for all $x \in [x_*-1,x_*]$.  This proves our claim.  

Figure 3 illustrates the two traveling wave solutions of the equation just described, with a small $\epsilon$.

\begin{center}
\includegraphics[width = 6in, height = 4in]{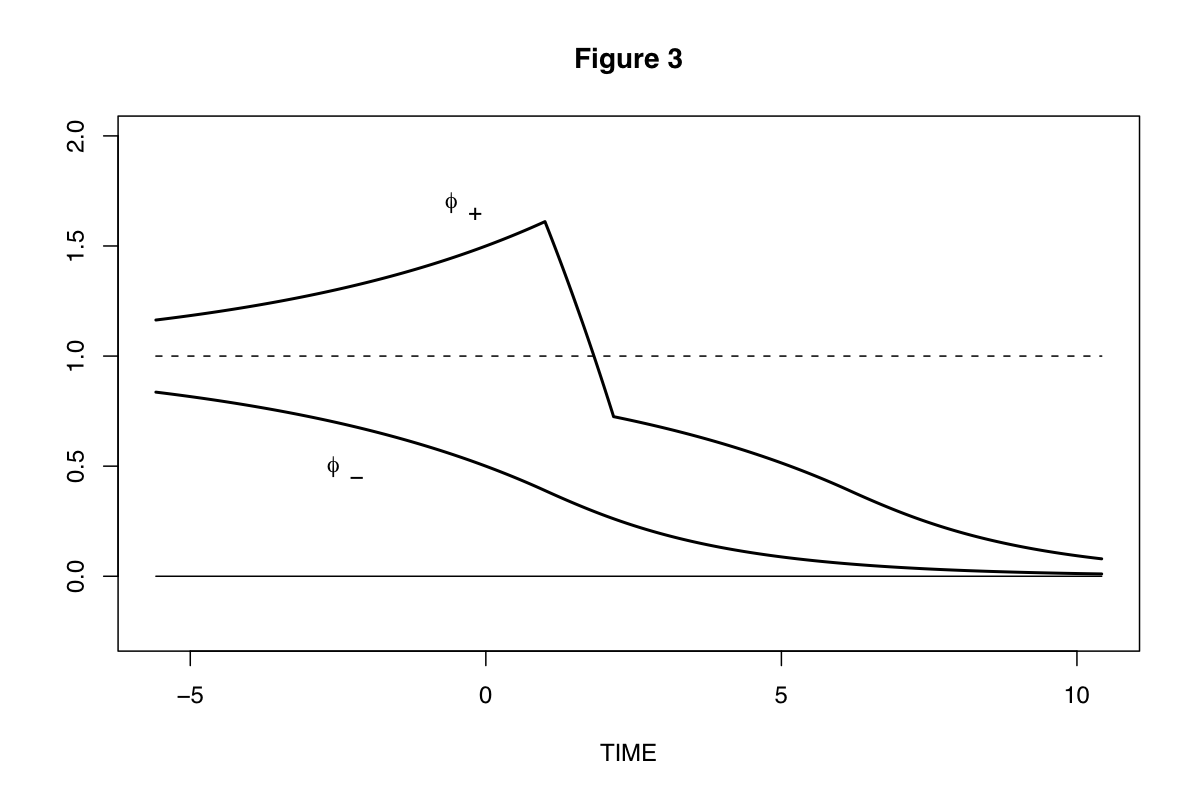}
\end{center}

\bibliography{TWrefs_ADHBBK}

\end{document}